\documentclass[11pt]{amsart}

\DeclareMathOperator{\Lie}{Lie}\DeclareMathOperator{\im}{im}
\DeclareMathOperator{\Rad}{Rad}
\DeclareMathOperator{\gr}{gr}
\DeclareMathOperator{\rk}{rk}
\DeclareMathOperator{\ad}{ad}

\usepackage[usenames,dvipsnames]{color}

\usepackage{amsmath,amssymb,amsthm,graphics,amscd}
\usepackage{longtable}
\usepackage{cite,xypic}
\usepackage{color}
\usepackage{graphicx}
\usepackage{epsf}
\usepackage{fancyhdr,hyperref}
\usepackage{lscape}
\hypersetup{backref,
colorlinks=true,backref,citecolor=Purple}

  \renewenvironment{thebibliography}[1]{
    \begin{oldthebibliography}{#1}
      \setlength{\parskip}{0ex}
      \setlength{\itemsep}{0ex}
  }
  {
    \end{oldthebibliography}
  }
  
\setlongtables
\begin{document}

\newcounter{rownum}
\setcounter{rownum}{0}
\newcommand{\ab}{\addtocounter{rownum}{1}\arabic{rownum}}

\newcommand{\x}{$\times$}
\newcommand{\bb}{\mathbf}

\newcommand{\Ind}{\mathrm{Ind}}
\newcommand{\Char}{\mathrm{char}}
\newcommand{\hra}{\hookrightarrow}
\newtheorem{lemma}{Lemma}[section]
\newtheorem{theorem}[lemma]{Theorem}
\newtheorem*{TA}{Theorem A}
\newtheorem*{TB}{Theorem B}
\newtheorem*{TC}{Theorem C}
\newtheorem*{CorC}{Corollary C}
\newtheorem*{TD}{Theorem D}
\newtheorem*{TE}{Theorem E}
\newtheorem*{PF}{Proposition E}
\newtheorem*{C3}{Corollary 3}
\newtheorem*{T4}{Theorem 4}
\newtheorem*{C5}{Corollary 5}
\newtheorem*{C6}{Corollary 6}
\newtheorem*{C7}{Corollary 7}
\newtheorem*{C8}{Corollary 8}
\newtheorem*{claim}{Claim}
\newtheorem{cor}[lemma]{Corollary}
\newtheorem{conjecture}[lemma]{Conjecture}
\newtheorem{prop}[lemma]{Proposition}
\newtheorem{question}[lemma]{Question}
\theoremstyle{definition}
\newtheorem{example}[lemma]{Example}
\newtheorem{examples}[lemma]{Examples}
\newtheorem{algorithm}[lemma]{Algorithm}
\newtheorem*{algorithm*}{Algorithm}
\theoremstyle{remark}
\newtheorem{remark}[lemma]{Remark}
\newtheorem{remarks}[lemma]{Remarks}
\newtheorem{obs}[lemma]{Observation}
\theoremstyle{definition}
\newtheorem{defn}[lemma]{Definition}

  \def\hal{\unskip\nobreak\hfil\penalty50\hskip10pt\hbox{}\nobreak
  \hfill\vrule height 5pt width 6pt depth 1pt\par\vskip 2mm}

\renewcommand{\labelenumi}{(\roman{enumi})}
\newcommand{\Hom}{\mathrm{Hom}}
\newcommand{\Int}{\mathrm{int}}
\newcommand{\Ext}{\mathrm{Ext}}
\newcommand{\opH}{\mathrm{H}}
\newcommand{\D}{\mathcal{D}}
\newcommand{\soc}{\mathrm{Soc}}
\newcommand{\SO}{\mathrm{SO}}
\newcommand{\Sp}{\mathrm{Sp}}
\newcommand{\SL}{\mathrm{SL}}
\newcommand{\GL}{\mathrm{GL}}
\newcommand{\OO}{\mathcal{O}}
\newcommand{\Y}{\mathbf{Y}}
\newcommand{\X}{\mathbf{X}}
\newcommand{\diag}{\mathrm{diag}}
\newcommand{\End}{\mathrm{End}}
\newcommand{\tr}{\mathrm{tr}}
\newcommand{\Stab}{\mathrm{Stab}}
\newcommand{\red}{\mathrm{red}}
\newcommand{\Aut}{\mathrm{Aut}}
\renewcommand{\H}{\mathcal{H}}
\renewcommand{\u}{\mathfrak{u}}
\newcommand{\Ad}{\mathrm{Ad}}
\newcommand{\N}{\mathcal{N}}
\newcommand{\Z}{\mathbb{Z}}
\newcommand{\la}{\langle}\newcommand{\ra}{\rangle}
\newcommand{\gl}{\mathfrak{gl}}
\newcommand{\g}{\mathfrak{g}}
\newcommand{\F}{\mathbb{F}}
\newcommand{\m}{\mathfrak{m}}
\renewcommand{\b}{\mathfrak{b}}
\newcommand{\p}{\mathfrak{p}}
\newcommand{\q}{\mathfrak{q}}
\renewcommand{\l}{\mathfrak{l}}
\newcommand{\del}{\partial}
\newcommand{\h}{\mathfrak{h}}
\renewcommand{\t}{\mathfrak{t}}
\renewcommand{\k}{\mathfrak{k}}
\newcommand{\Gm}{\mathbb{G}_m}
\renewcommand{\c}{\mathfrak{c}}
\renewcommand{\r}{\mathfrak{r}}
\newcommand{\n}{\mathfrak{n}}
\newcommand{\s}{\mathfrak{s}}
\newcommand{\Q}{\mathbb{Q}}
\newcommand{\z}{\mathfrak{z}}
\newcommand{\pso}{\mathfrak{pso}}
\newcommand{\so}{\mathfrak{so}}
\renewcommand{\sl}{\mathfrak{sl}}
\newcommand{\psl}{\mathfrak{psl}}
\renewcommand{\sp}{\mathfrak{sp}}
\newcommand{\Ga}{\mathbb{G}_a}

\newenvironment{changemargin}[1]{%
  \begin{list}{}{%
    \setlength{\topsep}{0pt}%
    \setlength{\topmargin}{#1}%
    \setlength{\listparindent}{\parindent}%
    \setlength{\itemindent}{\parindent}%
    \setlength{\parsep}{\parskip}%
  }%
  \item[]}{\end{list}}

\parindent=0pt
\addtolength{\parskip}{0.5\baselineskip}

\title{Maximal subalgebras of Cartan type in the exceptional Lie algebras}
\author{Sebastian Herpel}
\address{Ruhr Universit\"at Bochum \\ Bochum, Germany} 
\email{sebastian.herpel@rub.de {\text{\rm(Herpel)}}}

\author{David I. Stewart}
\address{King's College\\ Cambridge, UK} \email{dis20@cantab.net {\text{\rm(Stewart)}}}
\pagestyle{plain}
\begin{abstract}In this paper we initiate the study of the maximal subalgebras of exceptional simple classical Lie algebras $\g$ over algebraically closed fields $k$ of positive characteristic $p$, such that the prime characteristic is good for $\g$. In this paper we deal with what is surely the most unnatural case; that is, where the maximal subalgebra in question is a simple subalgebra of non-classical type. We show that only the first Witt algebra can occur as a subalgebra of $\g$ and give an explicit classification of when it is maximal in $\g$.\end{abstract}
\maketitle
{\small \tableofcontents}
\section{Introduction}

The maximal subalgebras of the simple finite-dimensional Lie algebras over the complex numbers were first established by Dynkin \cite{Dyn52}. 
The fact that in characteristic $0$ there is such a good correpondence between connected closed subgroups of simple algebraic groups and subalgebras of their Lie algebras means that this classification lifts readily to a classification of maximal closed connected subgroups of the corresponding simple algebraic groups over algebraically closed fields of characteristic $0$.
Gary Seitz took up the case of achieving a classification of connected maximal closed subgroups of the simple algebraic groups over fields of positive characteristic.
This was achieved in \cite{Sei87} for the classical algebraic groups and, under some fairly mild restrictions on the characteristic $p$ of $k$ in \cite{Sei91} for the exceptional algebraic groups. Later, in work by Liebeck and Seitz \cite{LS04}, the latter classification (for exceptional algebraic groups) was completed to cover all characteristics and extended to all maximal, closed, positive dimensional subgroups (not necessarily connected).
All these positive characteristic results rely on work of Donna Testerman \cite{Tes88}, \cite{Tes89}, \cite{Tes92} which, particularly, classify and construct subgroups of type $A_1$. The extension of the original work on subgroups of the classical groups continues to evolve. See \cite{BGT13} for the latest developments.

In this paper we consider the analoguous question for modular Lie algebras, a more direct analogue of Dynkin's original work. Apart from the intrinsic motivation to generalise Dynkin's results to positive characteristic, it is worth mentioning that maximal subalgebras of modular Lie algebras play an important role in the classification due to Premet and Strade \cite{PS06} of simple Lie algebras over algebraically closed fields of dimension $p>3$ as they give rise to Weisfeiler filtrations. Whereas the classification of simple algebraic groups by their root data is by now very well-documented, the classification of simple Lie algebras is highly non-trivial and is considered to be likely to be out of reach for the primes $p=2$ and $p=3$. Let us recall the result for $p>3$: a simple Lie algebra $L$ is either classical (i.e. it is the Lie algebra of a simple algebraic group, or a central quotient thereof); of one of the four families of Cartan type simple Lie algebras $W$, $K$, $S$ or $H$ (either graded, or in case $H$ or $S$, a filtered deformation); or $p=5$ and it is one of the Melikyan algebras.

Perhaps it is not surprising that the classification of maximal subalgebras of Lie algebras $\g$ over algebraically closed fields of characteristic at least $5$, even just those which come from algebraic groups $G$, is likely to be difficult. For instance, a fact one takes for granted when working with simple algebraic groups is a theorem of Borel and Tits which has as a corollary that all maximal non-reductive subgroups of $G$ are parabolic. But this is not true in general. For instance when $p|n$, the maximal non-semsimple subalgebras of the simple Lie algebra $\psl_n$ need not be parabolic. See O.~K.~Ten's work \cite{Ten87} for a classification of maximal non-semisimple subalgebras in the case that $\g$ is a classical Lie algebra of type $A$--$D$. The paper \cite{Ten87a} also gives a fairly coarse classification of the maximal semisimple subalgebra of these same Lie algebras. Lastly, it appears that the same author had at some point annouced a result classifying the maximal subalgebras of $\g_2$  when the characteristic is at least $5$, but this remains unpublished.

We should mention one other serious piece of work on maximal subalgebras of simple Lie algebras, due to H.~Melikyan \cite{Mel05}, who classifies in most cases, the maximal graded subalgebras of the Cartan type Lie algebras.

It is the point of this paper to initiate the study of maximal subalgebras of exceptional simple Lie algebras in good characteristic. Here one is fortunate that $p>3$ and so the Premet--Strade classification of simple Lie algebras holds. In this paper we are concerned with what is surely the most unnatural case; that is, where the maximal subalgebra in question is a simple subalgebra of non-classical type. 

The most straightforward non-classical algebra to describe is the first Witt algebra $W_1=W(1;1)$ of dimension $p$, the Lie algebra of derivations of the truncated polynomial ring 
$k[X]/\la X^p\ra$, where $p$ is the characteristic of $k$. When $p=2$ it is not simple, when $p=3$ it is isomorphic to $\sl_2$, but futher than that, it is simple and there are no more coincidences with other Lie algebras mentioned in the classification.
It has a basis $\{\del,X\del,\dots, X^{p-1}\del\}$, with structure constants given by $[X^i\del,X^j\del]=(j-i)X^{i+j-1}\del$.  The first Witt algebra does put in a number of guest appearances as subalgebras of $\g$ and there are precisely four occasions when it is maximal. The existence of $p$-subalgebras of type $W_1$ is essentially established by classifying the nilpotent element representing $\del$.

\begin{theorem}\label{necForW1s} 
Let $\g$ be a simple classical Lie algebra of exceptional type.
Suppose $W\cong W_1$ is a $p$-subalgebra of $\g$ and $p$ a good prime for $\g$. Let $\del\in W$ be represented by the nilpotent element $e \in \g$.
Then the following hold:
\begin{itemize}
\item[(i)] $e$ is a regular element in a Levi subalgebra $\l$ of $\g$ and the root 
system associated to $\l$ is irreducible.

\item[(ii)] For $h$ the Coxeter number of $\l$, we have either $p=h+1$ or $\l$ is of type $A_n$ and  $p=h$.

\item[(iii)] If $e$ is regular in $\g$ then $W$ is unique up to conjugacy.

\item[(iv)] If $e$ is regular and $\g$ is not of type $E_6$ then $W$ is maximal.

\item[(v)] If $e$ is not regular in $\g$ then $W$ normalises a non-trivial abelian subalgebra of $\g$, hence is 
not maximal.
\end{itemize}
Conversely, suppose that $e \in \g$ is nilpotent and $(e,p)$ satisfies the conditions
(i) and (ii) above. Then there exists a $p$-subalgebra isomorphic to $W_1$ with
$\del$ represented by $e$.
\end{theorem}

\begin{remarks}(i) In the statement of the theorem, recall that since $\g=\Lie(G)$ we have that $\g$ inherits a restricted structure, leading to a $p$-map; we say a subalgebra is a $p$-subalgebra of $\g$ if it is closed under this map. Now, since any subalgebra is an ideal in its $p$-closure and $\g$ is simple, all maximal subalgebras really are $p$-subalgebras.

(ii) If the nilpotent element $e$ is regular in a proper Levi subalgebra of $\g$ there can be many conjugacy classes of subalgebras of type $W_1$ containing $e$, in particular, those which are non-$G$-cr in the  sense of \cite{BMRT13}.

(iii) The reader is invited to notice the pleasant fact that a subalgebra isomorphic to $W_1$ is maximal only if $p|\dim \g$. \end{remarks}

In some sense it is an artefact of the large dimensions of non-classical simple Lie algebras in good characteristic that they cannot fit inside the exceptional Lie algebras. For example, the Melikyan algebras only exist when the characteristic of $k$ is $5$ and the smallest one is $125$-dimensional. Thus it cannot fit inside $G_2$, $F_4$ or $E_6$ and $p=5$ is not a good prime for $E_8$. So it remains to rule out the existence of a $125$-dimensional simple Lie algebra in $E_7$ of dimension just $133$, which is not too hard: see Lemma \ref{nomelikyanorcontact}. There could be more scope for finding, say the first Hamiltonian algebra of dimension $p^2-2$, but in fact we show that this never appears as a subalgebra of $\g$.

\begin{theorem}\label{onlyW1s}
Let $\g$ be a simple classical Lie algebra of exceptional type.
Suppose $p$ is a good prime for $\g$ and let $\h$ be a simple subalgebra of $\g$. Then $\h$ is either isomorphic to $W_1$ or it is of classical type.\end{theorem}

\begin{remark}The conclusion of Theorem \ref{onlyW1s} does not extend to bad characteristic. Alex Kubiesa, an undergraduate student of the second author, has discovered a maximal simple subalgebra of $F_4$ over $\mathbb F_3$ of dimension $26$. (And again, we have $26|\dim\g$.) There is strong evidence that this subalgebra is not isomorphic to the first contact algebra $K(3,[1,1,1])$. Of those known in characteristic $3$, it also matches the dimension of an Ermolaev algebra, see \cite[4.4]{Str04}.\end{remark}

\subsection*{Acknowledgements} We would like to thank Alexander Premet for a close reading of this paper and help with references on the non-graded Hamiltonians. We would also like to thank Dan Nakano for helpful discussions on the representations of Lie algebras of Cartan type.
\section{Preliminaries}
\subsection{Notation}
In the following $G$ will be a simple algebraic group of exceptional type over an algebraically closed field $k$ of characteristic $p$,
and $\g=\Lie(G)$ will be its Lie algebra.
We assume that $p$ is a good prime for the root system of $G$.

Fix a maximal torus $T$ and a Borel subgroup $B$ containing it and let $\Phi$ be the root system of $G$ corresponding to $T$, with positive roots $\Phi^+$ corresponding to $B$. If $S=\{\alpha_i\}$ represents the simple roots, one can express all other roots simply by giving the coefficients of the simple roots. We will use the Bourbaki ordering for this; hence the highest root of $F_4$, $2\alpha_1+3\alpha_2+4\alpha_3+2\alpha_4$ is written as $2342$ and the root $\alpha_2+\alpha_4$ in $E_6$ is written as $\begin{smallmatrix}0&0&1&0&0\\&&1&&\end{smallmatrix}$. 
We choose root vectors for $T$ in $\g$ and a basis for $\t=\Lie(T)$ coming from a basis of subalgebras isomorphic to $\sl_2$ corresponding to each of the simple roots. We write these elements as $\{e_\alpha:\alpha\in\Phi\}$ and $\{h_{\alpha}:\alpha\in S\}$ respectively.

\subsection{Nilpotent orbits}
We work extensively with nilpotent orbits in good characteristic. Our main source for the theory is \cite{Jan04}.
Let us recall the following facts from this reference, which we will usually use without comment. Associated to each nilpotent element $e\in \g$ is an orbit $\OO=G.e$ of $e$ under the adjoint action of $G$ on $\g$. We have $\dim G=\dim \OO+\dim G_e$ where $G_e$ is the centraliser of $e$ in $\g$. Since $p$ is a (very) good prime for $\Phi$, centralisers are smooth, and so $\dim G_e=\dim \g_e$; and $\Lie(G_e)=\g_e$. 
The nilpotent element $e$ is said to be \emph{distinguished} in some Levi subalgebra $\l=\Lie(L)$ of $\g$ if each torus in $L$ centralising $e$ is contained in $Z(L)$. Every nilpotent element is distinguished in at least one Levi subalgebra.
It is a result of Premet \cite{Pre95} that there is at least one cocharacter $\tau:\Gm\to G$ \emph{associated} to $e$. The cocharacter $\tau$ has the following properties:
firstly, $e$ is in the $2$-weight space for $\tau$, so that $\tau(t).e=t^2e$;
secondly $\tau$ evaluates in the derived subgroup $\D(L)$ of $L$,
where $L$ is a Levi subgroup with the property that $e$ is
distinguished in $\l = \Lie(L)$.
Any two associated cocharacters (with these properties) are conjugate by an element of $G_e$. Any cocharacter gives a grading of $\g=\bigoplus_{i\in \Z}\g(i)$ where $\g(i)$ is the $i$th weight space of $\tau$ on $\g$. One has $[\g(i),\g(j)]\subseteq \g(i+j)$. If $\tau$ is associated to $e$, one has $e\in\g(2)$ and $\g(\geq i):=\bigoplus_{i\geq 0}\g(i)$ is a parabolic subalgebra $\p=\Lie(P)$ of $\g$ with $\g(0)$ being a Levi subalgebra, and $\g(>0):=\bigoplus_{i>0}\g(i)$ is its nilradical, being $\g(>0)=\Lie(R_u(P))$. The $\tau$-grading on $\g$ induces a grading on the centraliser $\g_e$. One may write $G_e$ as a semidirect product $C_eR_e$ with $C_e$ reductive and $R_e$ its unipotent radical. In this case, one has $\g_e(0)=\Lie(C_e)$ and $\g_e(>0)=\Lie(R_e)$.

The classification of nilpotent orbits is now well-established; for detailed data in the case that $\g$ is exceptional and $p$ is good for $\g$, we are very grateful for the existence of \cite{LT11}, which gives complete tables of orbit representatives, associated cocharacters, and the explicit structure of $C_e$; that is, its root system, in terms of the roots of $G$ and $Z(C_e)^\circ$ in terms of the maximal torus $T$ of $G$. Furthermore the authors give the component group $C_e/C_e^\circ$ and the structure of $R_e$ in terms of modules for $C_e$. Since the calculation is used at one point, let us, by way of example of its usefulness, point out here that one can from such data read off the maximal value of $i$ for which $\g(i)\neq 0$. Let $i$ be this value: then since $[e,\g(i)]\subseteq \g(i+2)=0$, $\g(i)\subseteq\g_e(>0)$; thus $\g(i)$ is a $C_e$-module explicitly listed in \cite{LT11}.

\subsection{Representations of $W_1$} \label{repsofW1}
We assume that $p\geq 5$ in this section.
Let us recall some of the representation theory of $W_1$.
Recall that the Witt algebra $W_1$ can be given by a basis $\{\del,X\del,\dots,X^{p-1}\del\}$ with commutator formula
$[X^i\del,X^j\del]=(i-j)X^{i+j-1}\del$.
The simple modules for $W_1$ were determined in \cite{Cha41} and can be  obtained as quotients of Verma modules.
In this paper we are interested in $p$-representations, i.e.\ those
associated to the trivial central character in the universal enveloping algebra.
The relevant Verma modules are parametrised by the integers $\lambda$ from $0$ to $p-1$.
To describe these, let $\n^+=\la X^2\del,\dots,X^{p-1}\del\ra$ and $\b^+=\la X\del,\dots,X^{p-1}\del\ra$.
Then as $\n^+$ is an ideal of $\b^+$ we may define a $1$-dimensional $\b^+$-module $k_\lambda$ on which $\n^+$ acts 
trivially and $X\del$ acts by multiplication by $\lambda$.
Then one defines the corresponding Verma module $Z^+(\lambda)=\u(W_1)\otimes_{\u(\b^+)}k_\lambda$.
It is easy to see that $Z^+(\lambda)$ is $p$-dimensional with basis $\{m_0,\dots,m_{p-1}\}$
and the action of $e_k=X^{k+1}\del$ ($-1 \leq k \leq p-2$) is given by
\[e_k.m_j=(j+k+1+(k+1)\lambda)m_{j+k},\tag{*}\]
where we put $m_j = 0$ for $j$ outside $\{0,\dots,p-1\}$.
The $Z^+(\lambda)$ are all simple, except for $Z^+(0)$ and $Z^+(p-1)$.
The former has a trivial simple quotient and the latter has a trivial submodule and $p-1$-dimensional simple quotient.
We denote the corresponding simple quotient modules by $L(\lambda)$.

One way one may recognise the high weight of a simple module is the following:
In each $L=L(\lambda)$ there is, up to scalars, a unique vector, $m_0$ killed by $\del=e_{-1}$.
By formula (*) $X\del.m_0=e_0.m_0=(\lambda+1)m_0$ whenever $L(\lambda)$ is a Verma module. In the remaining two cases, one checks that for $L(p-1)$, $X\del$ has weight $1$ on a vector killed by $\del$ and of course $X\del$ has weight zero on the trivial module $L(0)$. Thus the action of $\del$ and $X\del$ on $L$ determine $L$ up to isomorphism. In particular, we may identify the adjoint module as $L(p-2)$: $\del$ is killed by $\ad\del$, and $[X\del,\del]=-\del$. Thus $\lambda+1=-1$ modulo $p$ and so $\lambda=p-2$.

If $V$ is a finite dimensional $W_1$-module, with the weights of $V$ known, and sufficently many maximal or primitive vectors for $\del$ are known, it is possible to list the composition factors of $V$. This works particularly in the case that $V$ is compatibly graded with $\del$ in grade $2$ and $X\del$ in grade $0$:

\begin{lemma} \label{lem:grade}
Suppose $V$ is a $W=W_1$-module admitting a grading $V=\bigoplus_{i\in\Z} V(i)$ such that $\del\cdot V(i)\subset V(i+2)$ and such that each $V(i)$ is stable under $X\del$. Then there exists a unique semisimple $W$-module $V_s=V_1\oplus V_2\oplus \dots \oplus V_r$ with $V_s=\bigoplus_{i\in\Z} V_s(i)$ with $V_s(i)=V(i)$ as $X\del$-modules and each $V_i$ is graded. 

For this module $V_s$, the set of composition factors $[V|W]$ and $[V_s|W]$ coincide.\end{lemma}
\begin{proof}If $V$ is irreducible then we simply take $V_s=V$ and $V_s(i)=V(i)$. Moreover there is only one choice for $V_s$ as a $W$-module since both $V$ and $V_s$ are determined by the maximum $r$ with $V(r)\neq 0$ and the weight of $X\del$ on this necessarily $1$-dimensional space.

Now take an irreducible submodule $V_1\cong L(\lambda)$ of $V$. Again, $V_1$ is determined uniquely as a $W$-module by the weight of $X\del$ on a vector killed by $\del$. Suppose this vector is $v=v_1+v_2+\dots+v_r$ with $v_j\in V(i_j)$, $i_1>i_2>\dots>i_r$. We may write $V_1=\la w=v\ra$ if  $\lambda=0$, $V_1=\la w,\del w, \dots, \del^{p-2}w=v\ra$ if $\lambda=p-1$ or $V_1=\la w,\del w, \dots, \del^{p-1}w=v\ra$ if $1\leq \lambda\leq p-2$. 

Now we regrade $V$. Let $V'=V$ as a $W$-module. Set $V(i)=V'(i)$ for $i>i_1$. If $\{v_1,x_2,\dots,x_m\}$ is a basis for $V(i_1)$ then let $V'(i_1)$ be spanned by $\{v,x_2,\dots,x_m\}$. Since $v$ is a weight vector for $X\del$ and $\del v=0\in V(i_1+2)$, for $i\geq i_1$, the new grading $V'(i)$ still satisfies the hypothesis of the lemma. Continuing inductively, if $\del^sw$ is in grade $V(t)$ for some $s\geq 0$ and some $t\in\mathbb Z$ we have $\del^{s-1}w$ a weight vector for $X\del$ with $\del\in V'(t)$ so we may grade $V'$ such that $V'(t-2)$ is spanned by $\del^{s-1}w$ together with some other vectors. In this new grading, we have $V_1$ is a graded submodule of $V'$.

Hence the quotient $V'/V_1$ is also graded, say $(V'/V_1)=\bigoplus (V'/V_1)(i)$. By induction, there exists a unique module $(V'/V_1)_s$ which is semisimple, with grading $(V'/V_1)_s=\bigoplus (V'/V_1)_s(i)$ satisfying $(V'/V_1)_s(i)=(V'/V_1)(i)$ as $X\del$-modules and with a decomposition into graded irreducibles and such that the $W$-composition factors of $V'/V_1$ are the same as those of $(V'/V_1)_s$. We therefore set $V_s=(V'/V_1)_s\oplus V_1$, with the direct factor $V_1$ graded as it is in $V'$. Moreover since the highest weight space of $V_1$ is determined by the top $i$ such that $V_1(i)\neq 0$ together with the weight of $X\del$ on this space, this is the unique choice of $V_s$ for which $V_s(i)=V(i)$ as $X\del$-modules for all $i$.\end{proof}

If we are in the situation of the lemma, we give an algorithm which produces the composition factors of $V$ given just the restriction of $V(i)$ to $X\del$, i.e. a list of $X\del$-weights $\ell_i$ of $V(i)$ for each $i\in\Z$. By the lemma, we may assume that $V=V_s$ satisfying the conclusions of the lemma.

\begin{prop}
Let $V$ be as in Lemma \ref{lem:grade}. For $i \in \Z$ with
$V(i) \neq 0$, let $\ell_i$ be a list (with multiplicities) of the $X\del$-weights
on $V(i)$. Then the following algorithm determines the composition
factors (with multiplicities) of $V$ as a $W$-module:

\begin{algorithm*}\label{wCompAlg}
\begin{enumerate}
\item Let $r \in \Z$ be maximal such that $\ell_r$ is nonempty. Pick $\mu \in \ell_r$.
\item Record a composition factor $U=L(\lambda)$ for $\lambda=\mu-1$ if $\mu\neq 0,1$ and $U=L(p-1)$, $L(0)$ if $\mu=1, 0$ respectively. Form a new set of lists $\{\ell'_r\}$ by removing weights from $\{\ell_r\}$ in the following way: If $U=L(0)$ remove a $0$-weight from $\ell_r$, if $U=L(p-1)$ remove one weight $1,2,\dots p-1$ from $\ell_r,\ell_{r-2},\dots,\ell_{r-2p+4}$ respectively and otherwise remove one weight $\mu,\mu+1,\dots,\mu+p-1$ from $\ell_r,\ell_{r-2},\dots,\ell_{r-2p+2}$. 
\item If the new lists $\{ \ell'_r \}$ are not all empty, repeat from Step (i). 
\end{enumerate}\end{algorithm*}
\end{prop}

\begin{proof}
By Lemma \ref{lem:grade}, we may assume that $V = \bigoplus_i V(i)$ is a semisimple $W$-module,
and that it is a direct sum of simple graded $W$-modules $V = \bigoplus_j V_j$.
Since we have $V(r) = \bigoplus_j V_j(r)$,
there exists a simple submodule $V_j$ such that $V_j(r) \neq 0$.
We have that $\del$ kills every element in $V(r)$,
hence the $X\del$-weight $\mu$ on $V_j(r)$ uniquely determines
the simple submodule $U=V_j$ as described in the algorithm,
and $V_j$ is a direct sum of one-dimensional graded
pieces in the positions given in step (ii) of the algorithm. 
Proceeding with $V/V_j$ in place of $V$, we may determine all composition factors
(with multiplicities) of $V$. Moreover, replacing $V$ by $V/V_j$ corresponds
to replacing the weights of $V$ by the weights obtained after a single application
of step (ii) in the algorithm.
\end{proof}

The self-dual simple modules are $L((p-1)/2)$, $L(0)$ and $L(p-1)$. Otherwise one has $L(i)^*=L(p-1-i)$. 

The extensions of the simple modules were determined in \cite{BNW09} and independently, in \cite{Ria11}.
When $p=2$, $W_1$ is no longer simple and when $p=3$, $W_1\cong \sl_2$ for which the answer is well known.

\begin{lemma}[{\cite[Theorem A,B]{BNW09}}]\label{ext1forw1}Let $p=5$. Then

\begin{enumerate}\item $\Ext^1_{W_1}(L(\mu),L(\lambda))\cong k$ if
\begin{enumerate}\item $\lambda-\mu=2,3$ (mod $p$), $1\leq \mu\leq p-2$, $1\leq \lambda\leq p-1$, or
\item $(\mu,\lambda)=(0,1)$, $(p-2,0)$, $(p-1,2)$ or $(p-1,3)$;\end{enumerate}
\item $\Ext^1_{W_1}(L(\mu),L(\lambda))\cong k\oplus k$ if $\{\mu,\lambda\} = \{0,p-1\}$;
\item $\Ext^1_{W_1}(L(\mu),L(\lambda))=0$ otherwise.\end{enumerate}

Let $p\geq 7$. Then 
\begin{enumerate}\item $\Ext^1_{W_1}(L(\mu),L(\lambda))\cong k$ if
\begin{enumerate}\item $\lambda-\mu=2,3,4$ (mod $p$), $1\leq \mu\leq p-2$, $1\leq \lambda\leq p-1$, or
\item $(\mu,\lambda)=(0,1)$, $(p-2,0)$, $(p-1,2)$, $(p-1,3)$ or $(p-1,4)$, or
\item $2\lambda^2-10\lambda+3=0$ (mod $p$), $\lambda-\mu=6$ (mod $p$), $1\leq \mu,\lambda\leq p-2$;\end{enumerate}
\item $\Ext^1_{W_1}(L(\mu),L(\lambda))\cong k\oplus k$ if $\{\mu,\lambda\}= \{0,p-1\}$;
\item $\Ext^1_{W_1}(L(\mu),L(\lambda))=0$ otherwise.\end{enumerate}
\end{lemma}

We record the next easy lemma which will be of use in proving Theorem \ref{necForW1s}.

\begin{lemma}\label{naturalreps}Suppose $p>2$. In each of its non-trivial irreducible $p$-representations, the smallest classical algebraic simple Lie algebra $\h$ containing $W_1$ is $A_{p-1}$ unless $V=L(p-1)$ and $\h=\sp_{p-1}$ or $V=L((p-1)/2)$ and $\h=\so_p$. Furthermore, the element $\del\in W_1$ is represented by a nilpotent regular element of $\h$.\end{lemma}
\begin{proof} By the earlier remarks of this section, a simple non-trivial $p$-representation $V$ is self-dual if and only if $V=L(p-1)$ or $V=L((p-1)/2)$. One checks directly, c.f. \cite[Lem. 11.5]{HS14}, that the action of $W_1$ on $L(p-1)$ preseves a symplectic form. Since the dimension of $L((p-1)/2)$ is odd, the action of $W_1$ must preserve an orthogonal form. Otherwise $V$ is not self-dual and $\dim V=p$, so the actions here give $W_1\subseteq\psl_p$ of type $A_{p-1}$. For the last statement, examining the action of $W_1$ on $V$ in each case, one sees that the element $\del$ acts in each case with a single Jordan block on $V$, i.e. on the natural module for $\h$. This shows that $\del$ is regular in $\h$.\end{proof}

\subsection{Representations of Hamiltonians} \label{repsofH2}
Again we assume that $p\geq 5$.
In Theorem \ref{onlyW1s} we claim that the only non-classical simple subalgebras of exceptional simple Lie algebras in good characteristic are isomorphic to $W_1$. Most possibilities can be ruled out on dimensional grounds, and we have a special argument to deal with the case of the Zassenhaus algebra $W(1;[2])$. It will remain to show that there are no subalgebras of $\g$ isomorphic to the first restricted (graded) Hamiltonian algebra $H_2=H(2;(1,1))^{(2)}$ of dimension $p^2-2$, the non-restricted (non-graded) Hamiltonian $H(2;(1,1);\Phi(\tau))^{(1)}$ of dimension $p^2-1$, the non-restricted (non-graded) Hamiltonian $H(2;(1,1);\Phi(1))=H(2;(1,1);\Delta)$ of dimension $p^2$ or the second Witt algebra $W_2=W(2;(1,1))$. For explicit descriptions of the Hamiltonian algebras and their minimal $p$-envelopes, see \cite[\S4.2]{Str04} or \cite[\S5]{FSW14}, \cite[\S10.3]{Str09} or \cite[\S5]{FSW14}, and \cite[\S10.4]{Str09} respectively. In \S\ref{sec:nonclass} we first show that there are no $p$-subalgebras $\g$ isomorphic to the minimal $p$-envelopes of these algebras. Since $H_2$ appears as a $p$-subalgebra of $W_2$ (indeed it is usually constructed in this way) it will suffice to show that there are no $p$-subalgebras $\h$ isomorphic to $H_2$, $H(2;(1,1); \Phi(\tau))^{(1)}$ or $H(2;(1,1);\Phi(1))$. And for this, we will show that there is no restriction of the adjoint module $\g|\h$ compatible with a further restriction to a chosen $p$-subalgebra isomorphic to $W_1$ in $H_2$. (Later, we will have computed the composition factors of the restriction of $\g$ to every possible $p$-subalgebra isomorphic to $W_1$.)

\subsubsection{Graded Hamiltonians.} Let us first recall a concrete description of $H:=H_2$ by basis and structure constants, which can be found in \cite{Kor78}, for example, or generated from the general description given in \cite[\S4]{SF88}. The algebra $H$ has a basis $e_{i,j}$ for $-1\leq i,j\leq p-2$ with $-1\leq i+j\leq 2p-5$. It is easy to check from this that $\dim H=p^2-2$. The multiplication in $H$ is given by the formula $[e_{i,j},e_{k,l}]=((i+1)(l+1)-(j+1)(k+1))e_{i+k,j+l}$ whenever $(i+k,j+l)$ satisfy the conditions above and $0$ otherwise.
The algebra $H$ is a restricted Lie algebra with
$p$-th powers are defined as $e_{0,0}^{[p]} = e_{0,0}$ and 
$e_{i,j}^{[p]} = 0$ otherwise.
Hence, $H$ has a grading $H=\bigoplus\sum_{i=-1}^{2p-5} H(i)$ (inherited from $W_2$) with $e_{i,j}$ in degree $i+j$. In particular $H(\geq 0):=\la e_{i,j}:i+j\geq 0\ra$ is a subalgebra of $H$ of index $2$. The subalgebra $H(\geq 0)$ is a semidirect product of the subalgebra $H(0)=\la e_{1,-1},e_{0,0},e_{-1,1}\ra\cong \sl_2$ and its $p$-ideal $H(>0)=\la e_{i,j}:i+j>0\ra$. 
As $[e_{0,0},e_{-1,1}]=2e_{1,-1}$ and $[e_{0,0},e_{1,-1}]=-2e_{1,-1}$ an $\sl_2$-triple is given by $(E,H,F)=(e_{-1,1},e_{0,0},e_{1,-1})$.

The elements $x:=e_{-1,0}$ and $y:=e_{0,-1}$ are important and span a vector space complement to $H(\geq 0)$ in $H$. Note that $[x,y]=0$. Since $H(>0)$ is an ideal in $H(\geq 0)$, any representation of $H(0)$ may be lifted to a representation of $H(\geq 0)$ by insisting that $H(>0)$ act trivially. Because $H(0)\cong \sl_2$, the irreducible $p$-representations $L(r)$ of $H(\geq 0)$ are classified by the integers $0\leq r\leq p-1$ with $L(r)$ of dimension $r+1$. Let us write $\hat L(r)$ for the corresponding module lifted to $H(\geq 0)$.

In the adjoint representation of $H$ on itself, the elements $e_{p-3,p-2}$  and $e_{p-2,p-3}$ span the highest weight space, i.e. the space killed by $H(>0)$. On these, the element $e_{0,0}$ has weight $(p-1)-(p-2)=1$ and $(p-2)-(p-1)=-1$. By Frobenius reciprocity it follows that the adjoint module is a quotient of the Verma module $M(1)=\Ind_{u(H(\geq 0),0)}^{u(H,0)}\hat L(1)$. (For more information on induced modules, see \cite[\S5.6]{SF88}.) The remaining $p$-representations of $H_2$ were first determined in \cite{Kor78}, though the most general reference, valid for higher rank Hamiltonian modular Lie algebras is \cite{She88} together with certain corrections made in \cite{Hol98}. For our purposes, the following statement is all we need:

\begin{lemma}\label{h2reps} A simple restricted representation of $H=H(2;(1,1))^{(2)}$ is isomorphic to one of $L_H(0)\cong k$, trivial; $L_H(1)$, the adjoint module; or the Verma module $L_H(r)=M(r)$ for $2\leq r\leq p-1$ of dimension $(r+1)p^2$ obtained by inducing the module $\hat L(r)$ from $H(\geq 0)$ to $H$.\end{lemma}

Now, let $W$ be the $p$-subalgebra of $H$ spanned by the elements $e_{0,j}$. It is easily seen that $W$ is isomorphic to the Witt algebra $W_1$ with the element $X\del$ represented by $h=e_{0,0}$ and the element $\del$ represented by $y=e_{0,-1}$. We wish to calculate the composition factors $[L_H(r)|W]$ of the simple restricted representations of $H$ to $W$ according to the process described in the previous section. 

\begin{lemma}\label{h2repsrest} The restrictions of simple restricted $H=H(2;(1,1))^{(2)}$-modules $L_H(r)$ to $W$ are as follows. We have $[L_H(0)|W]=L(0)$,  $[L_H(1)|W]=[\bigoplus_{r=1}^{p-2}L(r)\oplus L(p-1)^{ 2}]$, and \[[L_H(r)|W]=\left[\left(\bigoplus_{r=1}^{p-2}L(r)\oplus L(0)^2\oplus L(p-1)^2\right)^{(r+1)}\right].\] In particular every $p$-representation of $H$ restricted to $W$ contains the same number of composition factors of each $L(r)$ such that $1\leq r\leq p-2$.\end{lemma}
\begin{proof} The case $r=0$ is clear.
For $r=1$ notice that $\ad y(e_{a,-1})=[e_{0,-1},e_{a,-1}]=0$ for each $0\leq a\leq p-2$
and that $\ad h(e_{a,-1})=[e_{0,0},e_{a,-1}]=(-a-1)e_{a,-1}$. Thus $[L_H(1)|W]$ contains at least one copy of each composition factor $L(r)$ with $1\leq r\leq p-1$. The sum of the dimensions of these composition factors is $p^2-p-1$. Together these account for the full $0$-weight space of $h$ on $L_H(1)$; $X\del$ acts non-trivially on the remaining weight spaces. It follows that there is a further composition factor isomorphic to $L(p-1)$, which exhausts the dimension of $M$.

For the remaining cases $2\leq r\leq p-1$ we use the algorithm in Proposition \ref{wCompAlg}; though this requires some set-up.  We have $M(r)=\Ind_{H(\geq 0)}^H \hat L(r)$. We may take a basis $\{v_{r},v_{r-2},\dots,v_{-r}\}$ for $\hat L(r)$, where $v_i$ is in the $i$-weight space for $h$. Then since $\la x,y\ra$ is a vector space complement for $H(\geq 0)$ in $H$, we may take a basis $\{x^ay^b\otimes v_i: 0\leq a \leq p-1,\ 0\leq b\leq p-1,\ i=r-2c, 0\leq c\leq r\}$ of $M(r)$. The action of $z\in H$ on $M$ is given by $z.(x^ay^b\otimes v_i)=(zx^ay^b)\otimes v_i$. Every vector in this basis is a weight vector for $h$. Since $[x,y]=0$ we have $y.(x^ay^b\otimes v_i)=x^ay^{b+1}\otimes v_i$ and so (recalling $M(r)$ is a $p$-representation) each $x^ay^{p-1}\otimes v_i$ is killed by $y$. 

One checks that the span of the vectors $\{x^ay^b\otimes v_{r}:0\leq a,b\leq p-1\}$ is a $W$-submodule, $M(r)_{r}$; the key calculation here is that $e_{0,r}$ will commute with $x^ay^b$ in $u(H,0)$ modulo vectors $x,y$; or $e_{a,b}$ with $a+b>0$ together with $e_{-1,1}$, any of which kills $v_r$; or $e_{0,0}$, which stabilises $v_r$. Moreover we may grade $M(r)_{r}$ as $M(r)_{r}=\bigoplus M(r)_{r}(i)$ with $M(r)_{r}(2b)$ spanned by the vectors $\{x^ay^b\otimes v_{r}:0\leq a\leq p-1\}$. Then $M(r)_{r}$ satisfies the hypotheses of Proposition \ref{wCompAlg} and we may write down the composition factors according to that recipe. The weight of $h$ on the vector $x^ay^{p-1}\otimes v_{r}$ is $a-(p-1)+r=a+1+r$ for each $0\leq a\leq p-1$, thus there is a composition factor $L(a)$ for each $0\leq 1\leq p-1$. Removing the $h$-weights of these from $M(r)_{r}$ and repeating, we find additionally a copy of the $W_1$-modules $L(0)$ and of $L(p-1)$. In the quotient of $M(r)$ by $M(r)_{r}$ we have a submodule $M(r)_{r-2} + M(r)_{r}\subset M(r)/M(r)_{r}$ on which we may repeat the same task. Since there are $r+1$ values of $i$ on which we perform this task, we are done.\end{proof}

\subsubsection{Non-graded Hamiltonians.} Essentially the same task can be performed for the minimal $p$-envelopes of the simple Lie algebras $H(2;(1,1);\Phi(\tau))^{(1)}$ and $H(2;(1,1);\Phi(1))$. By \cite[\S10.3, \S10.4]{Str09} the minimal $p$-envelope of each is $(p^2+1)$-dimensional. There is again a co-dimension $2$ subalgebra $H(\geq 0)$ containing  $H(>0)$ of dimension $p^2-4$ as an ideal and with $H(\geq 0)/H(>0)\cong \sl_2$. As before, we lift all simple restricted representations from $H(0)$ to $H(\geq 0)$ by letting $H(>0)$ acts trivially and induce the Verma modules $M(r)$ from $H(\geq 0)$ to $H$.

We will want to see that the only reducible Verma modules are $M(1)$ and $M(0)$. This is supplied by \cite[Theorem 5.3]{FSW14} for $H(2;(1,1);\Phi(\tau))^{(1)}$, but we will need to follow the same line of argument in the other case. For this we need an explicit description of the elements of $H=H(2;(1,1);\Phi(1))$ in terms of elements of $W_2$. The latter has basis $\{X^iY^j\del_X,X^iY^j\del_Y:0\leq i,j\leq p-1\}$ and graded with $X^iY^j\del_X$ and $X^iY^j\del_Y,$ in degree $i+j-1$. By \cite[\S10.4]{Str09}, $H$ is spanned by the elements $\del_X(f)\del_Y-\del_Y(f)\del_X-X^{p-1}f\del_Y$ for $f\in k[X,Y]/(X^p,Y^p)$. Applying this recipe to the monomial $f=X^iY^j$, we see that $H$ has basis \[\{jY^{j-1}\del_X+X^{p-1}Y^j\del_Y,iX^{i-1}Y^j\del_Y-jX^iY^{j-1}\del_X:1\leq i\leq p-1, 0\leq j\leq p-1\}\]or in divided power notation, used for example by GAP,
\[\{Y^{(j-1)}\del_X-X^{(p-1)}Y^{(j)}\del_Y,X^{(i-1)}Y^{(j)}\del_Y-X^{(i)}Y^{(j-1)}\del_X:1\leq i\leq p-1, 0\leq j\leq p-1\},\]
where $X^{(-1)}=Y^{(-1)}=X^{-1}=Y^{-1}$ is understood to be zero.

Only the element $\del_X-X^{(p-1)}Y\del_Y$ has a $p$th power outside this set, viz. $-Y\del_Y$, so that adding for example the element $X\del_X+Y\del_Y$ to this basis gives the basis of the minimal $p$-envelope $L$ of $H$.

Since $L$ is a $p$-subalgebra of $W_2=W(2;(1,1))$, we induce a restricted descending filtration on $L$ from the natural grading $W_2=\bigoplus_{d=-1}^{2p-3}W_d$, namely $L_{(n)}:=L\cap W(2;(1,1))_{(n)}$, where $W(2;(1,1))_{(n)}=\bigoplus_{d\geq n}W(2;(1,1))_d$. One checks this filtration has depth $1$ and height $2p-4$: \emph{a fortiori} we have $W(2;(1,1))_{r}=0$ for $r\leq -2$ or $r\geq 2p-2$, and $W(2;(1,1))_{(2p-3)}=W(2;(1,1))_{2p-3}$ is spanned by the two elements $X^{p-1}Y^{p-1}\del_X$  and $X^{p-1}Y^{p-1}\del_Y$ hence has no intersection with $L$. We claim that the associated graded algebra $\gr L$ is isomorphic to $H(2;(1,1))$. A basis of the latter is \[\{X^{(i-1)}Y^{(j)}\del_Y-X^{(i)}Y^{(j-1)}\del_X:0\leq i,j\leq p-1, (i,j)\neq (0,0)\}\cup\{X^{(p-1)}\del_Y,Y^{(p-1)}\del_X\}\] and so we define a linear map from $L$ to $H(2;(1,1))$ which is an identity on their intersection in $W(2;(1,1))$ and where we send the basis element $Y^{(j-1)}\del_X-X^{(p-1)}Y^{(j)}\del_Y\mapsto Y^{(j-1)}\del_X$. It is clear this descends to an isomorphism of restricted graded Lie algebras $\gr(L)\to H(2;(1,1))$.

We have the ingredients to apply the following theorem, which we will do in the succeeding lemma.
\begin{theorem}[{\cite[Thm.~4.3]{FSW14}}]\label{fswthm}Let $L$ be a restricted Lie algebra, and let $(L_{(n)})_{n\in\Z}$ be a descending restricted filtration of $L$ of depth $1$ and height $h$. Let $V$ be a restricted $\gr_0(L)$-module, and let $M:=\Ind_{(u(L_{(0)}),0)}^{(u(L),0)}(V)$. Then one has a canonical isomorphism of graded restricted $\gr(L)$-modules \[\Ind_{u(\gr_+(L),0)}^{u(\gr(L),0)}(V)\cong \gr M,\]
where $\gr_+(L):=\bigoplus_{n\geq 0}\gr_n(L)$. In particular, if $\Ind_{u(\gr_+(L),0)}^{u(\gr(L),0)}(V)$ is an irreducible $\gr(L)$-module, then $M$ is an irreducible $L$-module.\end{theorem}

\begin{lemma}\label{weirdh2reps}A simple restricted representation of $H$, the minimal $p$-envelope of either the algebra $H(2;(1,1);\Phi(\tau))^{(1)}$ or $H(2;(1,1);\Phi(1))$, is isomorphic to one of $L_H(0)\cong k$, trivial; $L_H(1)$, the adjoint module of dimension $p^2-1$ or $p^2$, respectively; or $L_H(r)$ for $2\leq r\leq p-1$, the irreducible Verma module of dimension $(r+1)p^2$.\end{lemma}
\begin{proof}A standard argument using Frobenius reciprocity gives every restricted simple module as a quotient of a restricted Verma module. As for the case $H(2;(1,1))^{(2)}$ it is straightforward to identify the adjoint module as a quotient of $M(1)$. Thus it suffices to show that the Verma modules $M(r)$ for $2\leq r\leq p-1$ are all irreducible. The case where $H$ is the minimal $p$-envelope of $H(2;(1,1);\Phi(\tau))^{(1)}$ is given by \cite[Theorem 5.3]{FSW14} whose line of argument we follow for the case $H(2;(1,1);\Phi(1))$. 

From the remarks above, we have $\Y:=\gr L\cong H(2;(1,1))$, which contains the simple graded subalgebra $\X=H(2;(1,1))^{(2)}$ with the cokernel of the map $\X\to \Y$ concentrated in grades $p-2$ and $2p-4$. Hence the canonical map \[\Ind_{u(\X_{(0)},0)}^{u(\X,0)}(\hat L(\lambda)|_\X)\to \Ind_{u(\Y_{(0)},0)}^{u(\Y,0)}(\hat L(\lambda))|_\X\] is an isomorphism. As $\Ind_{u(\X_{(0)},0)}^{u(\X,0)}(\hat L(\lambda)|_\X)$ is an irreducible restricted $\X$-module for $\lambda_0\neq 0,1$ by \cite{Hol98} this implies that $\Ind_{u(\Y_{(0)},0)}^{u(\Y,0)}(\hat L(\lambda))$ is an irreducible $\Y$-module. Consequently, by Theorem \ref{fswthm}, $M(\lambda)$ is an irreducible $L$-module unless $\lambda=0,1$.
\end{proof}

We wish to restrict each simple module to a suitable subalgebra of each non-graded Hamiltonian which is isomorphic to $W_1$. Thus to play the same game as before, we need in both case a $p$-subalgebra $W$ isomorphic to $W_1$. Such subalgebras do not appear to be well-known. Using the notation of \cite[\S5]{FSW14} for the case $H(2;(1,1);\Phi(\tau))$, the following lemma gives us such a subalgebra. The proof is left to the reader.
\begin{lemma}\label{w1subinweirdh} The subalgebra $H(2;(1,1);\Phi(\tau))^{(1)}$ of $W_2$ contains a $p$-subalgebra $W=W_1$ having basis \[\{(1-X^{(p-1)}Y^{(p-1)})\del_X,X\del_X-Y\del_Y, X^{(2)}\del_X-XY\del_Y, X^{(3)}\del_X-X^{(2)}Y\del_Y,\dots, X^{(p-1)}\del_X - X^{(p-2)}Y\del_Y\}\]
with these elements playing the roles of $\del,X\del,X^2\del,\dots,X^{p-1}\del$, respectively.

The subalgebra $H(2;(1,1);\Phi(1))$ of $W_2$ contains a $p$-subalgebra $W=W_1$ having basis \[\{\del_Y,Y\del_X-X\del_Y,Y^{(2)}\del_Y-XY\del_X,Y^{(3)}\del_Y-XY^{(2)}\del_X\dots, Y^{(p-1)}\del_Y-XY^{(p-2)}\del_X\}\]
with these elements playing the roles of $\del,X\del,X^2\del,\dots,X^{p-1}\del$, respectively.
\end{lemma}

Finally, the same technique used before yields the following.

\begin{lemma}\label{weirdh2repsrest}The restrictions of simple restricted modules   $L_H(r)$ for $H=H(2;(1,1);\Phi(\tau))^{(1)}$ or $H(2;(1,1);\Phi(1))$ to the subalgebra $W$ provided by Lemma \ref{w1subinweirdh} are as follows. We have $[L_H(0)|W]=L(0)$,  $[L_H(1)|W]=[\bigoplus_{r=0}^{p-2}L(r)\oplus L(p-1)^{ 2}]$ for $H=H(2;(1,1);\Phi(\tau))^{(1)}$,  $[L_H(1)|W]=[\bigoplus_{r=1}^{p-2}L(r)\oplus L(p-1)^{ 2}\oplus k^2]$ for $H(2;(1,1);\Phi(1))$ and \[[L_H(r)|W]=\left[\left(\bigoplus_{r=1}^{p-2}L(r)\oplus L(0)^2\oplus L(p-1)^2\right)^{(r+1)}\right].\] In particular every $p$-representation of $H$ restricted to $W$ contains the same number of composition factors of each $L(r)$ such that $1\leq r\leq p-2$.\end{lemma} 
\begin{proof}The case $r=0$ is easy. For $r=1$, we analyse the action of $h=X\del_X-Y\del_Y$ on vectors killed by $\del$ in each case. 

In the first case, $\del=(1-X^{(p-1)}Y^{(p-1)})\del_X$. Each of the vectors $Y^r\del_X$ for $0\leq r\leq p-2$ in $H$ are killed by $\del$ and one gets one of each weight from $0$ to $p-2$. Thus there must be at least one composition factor of each type in $L_H(1)|W$. This accounts for $p^2-p$ dimensions of $L_H(1)$ and there remain $p-1$ to find. However, the $0$ weight space for $h$ is accounted for so we must have one copy of $L(p-1)$ remaining. 

In the second case, each of the vectors $X^r\del_Y$ for $0\leq r\leq p-1$ are killed by $\del$ with one of each weight from $0$ to $p-2$ occurring. Further, one checks that the span of $\{X^{(p-1)}\del_Y,X^{(p-1)}Y\del_Y-\del_X,X^{(p-1)}Y^{(2)}\del_Y-Y\del_X,\dots,X^{(p-1)}Y^{(p-1)}\del_Y-Y^{(p-2)}\del_X\}$ is a $p$-dimensional $W$-submodule of $\g$. This contains a further  trivial submodule spanned by the vector $X^{(p-1)}\del_Y$, hence has the structure $L(p-1)/k$, isomorphic to $O_1$ as a $W_1$-module. Counting up the dimensions now found, there is just one left, which must correspond to a trivial composition factor.

The case of the Verma modules $M(r)$ is similar to that in Lemma \ref{h2repsrest}.\end{proof}

\subsection{GAP calculations} \label{GAP-calcs}
At various points, notably for some intricate calculations performed in Appendix \ref{appendix}, we use the routines included in the standard GAP distribution for computing with Lie algebras. See \cite[Manual, Chapter 64]{GAP} for complete details. While many, possibly all of our calculations could be attempted by hand, this reduces time and error. 

A very straightforward use of GAP we employ is to use its database of root systems to evaluate a cocharacter on a root, thereby finding the weight of the cocharacter on a corresponding root vector for which we have written a function {\tt findweights}. For example, if $\g=E_6$, $\OO=A_4$ then \cite{LT11} notates an associated cocharacter in terms of certain elements of a maximal torus as $\tau:\begin{smallmatrix}2&2&2&-6&0&\\&&2\end{smallmatrix}$. Thus we may compute
\begin{verbatim}gap> T:=[2,2,2,2,-6,0];                                         
[ 2, 2, 2, 2, -6, 0 ]
gap> findweights(T,4);
[ [ 1, 0, 1, 0, 0, 0 ], [ 0, 1, 0, 1, 0, 0 ], [ 0, 0, 1, 1, 0, 0 ], 
  [ 1, 1, 1, 2, 1, 0 ], [ 1, 1, 1, 2, 1, 1 ], [ 1, 2, 2, 3, 2, 1 ], 
  [ 0, 0, 0, -1, -1, 0 ], [ 0, 0, 0, -1, -1, -1 ], [ 0, -1, -1, -2, -2, -1 ] ]\end{verbatim}
Thus one concludes that $\dim\g(4)=9$ and a basis of $\g(4)$ is \[\left\{e_{\begin{smallmatrix}1&1&0&0&0\\&0&&&\end{smallmatrix}},
e_{\begin{smallmatrix}0&0&1&0&0\\&1&&&\end{smallmatrix}},
e_{\begin{smallmatrix}0&1&1&0&0\\&0&&&\end{smallmatrix}},
e_{\begin{smallmatrix}1&1&2&1&0\\&1&&&\end{smallmatrix}},
e_{\begin{smallmatrix}1&1&2&1&1\\&1&&&\end{smallmatrix}},
e_{\begin{smallmatrix}1&2&3&2&1\\&2&&&\end{smallmatrix}},
e_{-\begin{smallmatrix}0&0&1&1&0\\&0&&&\end{smallmatrix}},
e_{-\begin{smallmatrix}0&0&1&1&1\\&0&&&\end{smallmatrix}},
e_{-\begin{smallmatrix}0&1&2&2&1\\&1&&&\end{smallmatrix}}\right\}.\]

Very frequently, we will wish to compute the Lie bracket of two expressions of the form $\sum_\alpha\lambda_\alpha e_\alpha+\sum_\beta\mu_\beta h_\beta$, where the $e_\alpha$ are root vectors in $\g$, and the $h_\beta$ are elements from a basis of a maximal torus defining the root system of $\g$. The $\lambda_\alpha$ and $\mu_\beta$ are treated as scalar indeterminates. For example, if $x$ is an expression of the above form, and $y$ is some fixed element, then calculating $[x,y]$ and insisting that it is zero will put conditions amongst the $\lambda_\alpha$ and $\mu_\beta$. (Such calculations can of course be done by hand, since the bracket of any pair of basis elements is given by structure constants, which can be deduced from the root system.) To do this in GAP we set up a polynomial ring $R$ in a large enough number (e.g. $\dim\g$) indeterminates and then work with $\g(R)$. Note that in GAP, the `canonical' basis $B$ of a simple (classical) Lie algebra $\g$ in GAP is arranged so that (i) the last $\rk\g$ elements are a basis for a maximal torus; (ii) the first $|\Phi^+|$ elements are positive root vectors; (iii) the first $\rk\g$ elements are simple root vectors; (iv) if $r\leq |\Phi^+|$ then the $r+|\Phi^+|$th element of $B$ is a root vector corresponding to the negative of the root corresponding to the $r$th element of $B$.

For example if $\g=E_8$, and $e$ is a nilpotent element of type $A_4$, then the following calculates the bracket of $e$ with a general element of the maximal torus $\mu_1h_1+\dots +\mu_8h_8\in\t$.
\begin{verbatim}gap> g_rat := SimpleLieAlgebra("E",8,Rationals);;
gap> R := PolynomialRing(Rationals,Dimension(g_rat));;
gap> g := SimpleLieAlgebra("E",8,R);;
gap> x := IndeterminatesOfPolynomialRing(R);;
gap> B := Basis(g);;
gap> e:=B[1]+B[2]+B[3]+B[4]; 
v.1+v.2+v.3+v.4
gap> y:=x[241]*B[241]+x[242]*B[242]+x[243]*B[243]+x[244]*B[244]+x[245]*B[245]
> +x[246]*B[246]+x[247]*B[247]+x[248]*B[248];
(x_241)*v.241+(x_242)*v.242+(x_243)*v.243+(x_244)*v.244+(x_245)*v.245+(x_246)*\
v.246+(x_247)*v.247+(x_248)*v.248
gap> e*y;
(-2*x_241+x_243)*v.1+(-2*x_242+x_244)*v.2+(x_241-2*x_243+x_244)*v.3+(x_242+x_2\
43-2*x_244+x_245)*v.4
\end{verbatim}
We have also implemented a routine which will make substitutions in general elements in order to make a certain expression be zero. For example, if one wanted to calculate $\c_\t(e)$, one could insist that the last expression in the above output be zero. Thus we might choose to calculate the substitution into $\tt y$ of $\tt x\_243=2*x\_241$, $\tt x\_244=2*x\_242$ and so on, a process which this algorithm automates.

\section{Finding $W_1$ subalgebras: Proof of Theorem \ref{necForW1s}.}
Recall that $G$ is an exceptional simple algebraic group over an algebraically closed field of good characteristic $p$, with $\g$ its Lie algebra.
In this section (together with Appendix \ref{appendix}) we prove Theorem \ref{necForW1s} by means of a series of lemmas. Proposition \ref{findingcoch} will be quite central in proving the conjugacy statements involved: that is part (iv) of Theorem \ref{necForW1s}. For this we need the following two lemmas.

\begin{lemma}\label{findingtorus}

Let $e\in \g=\Lie(G)$  be a $p$-nilpotent element and let $\h:=\la e \ra$ be the subspace it generates. Then any torus $\c\subseteq\g$ normalising $\h$ can be written $\c=\Lie(C)$ for $C$ a torus of $G$ normalising $\h$.\end{lemma}

\begin{proof}

We first prove that the normaliser $N_{G}(\h)$ of $\h$ in $G$ is smooth: 
By the existence of associated cocharacters, there is a one-dimensional torus
$S$ which normalises $\h = \la e\ra$ but does not centralise it.
Differentiating this torus, we get also a $1$-torus $\s\subseteq\g$ which normalises $e$ but does not centralise $e$.
We may calculate the dimension of $\n_{\g}(\h)$ by looking at its action on $\h$. By rank--nullity, we have $\dim \n_{\g}(\h)=\dim \c_{\g}(\h)+1$. Equally, we may calculate the dimension of $N_{G}(\h)$ by looking at its action on $\h$. Again, we have $\dim N_{G}(\h)=\dim C_{G}(\h)+1$.
But since $p$ is very good, $C_G(\h) = G_e$ is smooth, so $\dim C_{G}(\h)=\dim \c_{g}(\h)$. 
Hence the dimensions of the group--theoretic and Lie--theoretic normalisers coincide and the normaliser is smooth as required.

If $H$ is any smooth algebraic group, then \cite[Theorem 13.3]{Hum67} shows that any maximal torus $\t$ of $\h$ can be written $\Lie(T)$ for $T$ a maximal torus of $H$.
In particular, this applies to $N_G(\h)$.
We may choose an embedding $N_G(\h) \subseteq \GL_n$ with $\c\subseteq\t$ diagonal.  Now \cite[Prop.\ 2]{Dieu:1953} gives that $\c=\Lie(C)$ for $C\subseteq T\subseteq N_G(\h)$.
\end{proof}

\begin{lemma}\label{adecapge}Suppose $e$ is a nilpotent element in $\g$, distinguished in a Levi subalgebra $\l = \Lie(L)$.
Then $\im\ad e\cap\g_e(0)=0$ unless $L$ has a factor of type $A_{p-1}$.
In the remaining eight orbits, the intersections are given in Table \ref{t:adecapg}.
\begin{table}\begin{center}\begin{tabular}{l|l|l|l}
$\g$ & $p$ & $\OO$ & $\g_e(0)\cap\g$\\\hline
$E_6$ & $5$ & $A_4$ & $\z(\l')$\\
$E_6$ & $5$ & $A_4A_1$ & $\z(\l')$\\
$E_7$ & $5$ & $A_4$ & $\z(\l')$\\
$E_7$ & $5$ & $A_4A_1$ & $\z(\l')$\\
$E_7$ & $5$ & $A_4A_2$ & $\g_e(0)\cong A_1$\\
$E_7$ & $7$ & $A_6$ & $\g_e(0)\cong A_1$\\
$E_8$ & $7$ & $A_6$ & $A_1\subseteq \g_e(0)\cong A_1^2$\\
$E_8$ & $7$ & $A_6A_1$ & $\g_e(0)\cong A_1$
\end{tabular}\end{center}\caption{\label{t:adecapg}Non-trivial intersections of $\im\ad e$ with $\g_e(0)$.}\end{table}\end{lemma}

\begin{proof} Unless $\OO$ is one of the exceptional orbits, this is stated in \cite[p57]{Jan04}. For the remainder, one simply takes a general element in $\g(-2)$ and applies $e$ to it to get a general element $v$ in $\im \ad e\cap\g(0)$. Then insisting that the general element satisfies $[e,v]=0$ gives the data above.
See Section \ref{GAP-calcs} on how such calculations can be carried out
with the help of GAP.\end{proof}

We are now in a position to prove the key result about finding suitable
cocharacters associated to a nilpotent element $e$.
The toral element $J$ will later be taken to equal $X\del \in W_1$.

\begin{prop}\label{findingcoch} Suppose $e\in\g$ is a nilpotent element and let $J=J^{[p]}$ be a toral element of $\g$ normalising but not centralising $\la e\ra$. Let $\chi$ be an associated cocharacter to $e$ and let $\g(i)$ be the associated $i$-th graded piece of $\g$. Then the following hold:
\begin{itemize}
\item[(i)] $J$ is conjugate by an element $g$ of $R_e=R_u(G_e)$ to an element of $\g(0)$; thus replacing $\chi$ by its conjugate by $g^{-1}$ and taking the new associated grading $\g(i)$, we may assume $J$ normalises $\g(i)$ for each $i\in\Z$. 

\item[(ii)] There exists a cocharacter $\tau$ associated to $e$ and a toral element
$H$ with $\Lie(\tau(\Gm)) = \la H \ra$ such that
$J = H + H_0$ for some toral element $H_0 \in \g_e(0)$.

\item[(iii)]  Suppose $e$ is not in an orbit containing a factor of type $A_{p-1}$, and that $J$ is in the image of $\ad e$. Then there is a cocharacter $\tau$ associated to $e$ with $\Lie(\tau(\Gm))=\la J\ra$.
\end{itemize}
\end{prop}
\begin{proof}By Lemma \ref{findingtorus} there is a torus $T_1\leq G$ such that $\Lie(T_1)=\la J\ra$, with $T_1$ normalising $\la e\ra$.
Since $J$ acts non-trivially on $e$ so does $T_1$. Let $T_2=\chi(\Gm)$.
Then for each $t_1\in T_1$ there exists $t_2\in T_2$ such that $t_1=t_2.s$ with $s\in G_e$. Thus $T_1G_e=T_2G_e$. Write $G_e=C_eR_e$ with $C_e$ reductive.
Then $T_2C_e$ is a subgroup of $T_2G_e$ and we may take the image $\bar T_1$ of $T_1$ in $T_2C_e$ under the projection $T_2C_eR_e\to T_2C_e$.
Now $T_1\cap R_e=\{1\}$ so that $T_1 \subseteq \bar T_1 R_e$
is a complement to $R_e$.
Since $R_e$ is unipotent, $T_1$ is conjugate to $\bar T_1$ by an element of $R_e$. Thus $\Lie(T_1)=\la J\ra$ is conjugate to a subalgebra of $\Lie(T_2C_e)\leq \g(0)$ as required. For the last part of (i) observe that $J^g\in\g(0)$ means that $\chi(t)gJg^{-1}\chi(t)^{-1}=gJg^{-1}$. Thus $\chi(t)^{g^{-1}}J\chi(t)^{-g^{-1}}=J$. This proves (i).

For (ii), by part (i) we may assume that $J$ belongs to $\g(0)$.
Take $T_2$ as above and let $\Lie T_2=\la H\ra$ with $H$ chosen so that $[H,e]=[J,e]$.
Thus $J-H$ is an element in the zero-grade of the centraliser $\g_e(0)$,
so we may write $J=H+H_0$ with $H_0\in\g_e(0)$.

Hence $\tau = \chi$ satisfies the assertions in (ii).

For (iii), we claim that $H$ and $H_0$ are in the image of $\ad e$. For $H$, we may consider an optimal $\SL_2$-homomorphism $\varphi: \SL_2 \rightarrow G$ associated to $e$, see \cite[Prop.\ 33]{McN05}. By definition, $d \varphi$ maps the nilpotent
$\left ( \begin{smallmatrix} 0 & 1 \\ 0 & 0 \end{smallmatrix} \right )$ onto $e$, and the
map $t \mapsto \varphi 
\left ( \begin{smallmatrix}
t & 0 \\
0 & t^{-1} \end{smallmatrix} \right )$ is a cocharacter associated to $e$. 
Temporarily replacing everything by a conjugate by an element of $G_e$ we may assume that this cocharacter is equal to $\chi$. Thus $\la H\ra=\Lie(\chi(\Gm))$ and $H$ is in the image of $\ad e$. Thus $H_0$ is in the image of $\ad e$ also.  

Now, since the orbit type of $e$ does not have a factor of type $A_{p-1}$, by Lemma \ref{adecapge} we must have $H_0=0$. Thus $J=H$ and we may take $\tau=\chi$. This proves (iii).

\end{proof}

\begin{lemma} \label{e-candidates} 
Let $e\in\g$ be a nilpotent element, let $\tau$ be an associated cocharacter and 
$\g(k)$ the $k$th graded piece of $\g$ associated to $\tau$. 
Suppose $\g$ contains a $p$-subalgebra $W\cong W_1$ with $e=\del$. 
Then the following two conditions must be satisfied:

\begin{enumerate}
\item We have $e^{[p]}=0$;
\item The maximal value of $k$ with $\g(-k)\neq 0$ satisfies $2p-4 \leq k \leq 2p-2$.
\end{enumerate}
\end{lemma}

\begin{proof}
The necessity of (i) is clear, since $W$ is assumed to be a $p$-subalgebra of $\g$ and the condition $\del^{[p]}=0$ holds in $W$.

For (ii), the upper bound follows from (i) and \cite[Prop.\ 30]{McN05}.
For the lower bound, observe that $e$ must be a non-zero vector in the image of $\ad(e)^{p-1}$. 
This implies that there is a vector $f$ in $\g(-2p+4)$
for which $\ad(e)^{p-1}(f) = e$ is non-zero.
\end{proof} 

The next lemma proves the first two parts of Theorem \ref{necForW1s}. We perform several case by case checks on nilpotent elements satisfying the conclusions of Lemma \ref{e-candidates} in order to check that the relevant statements hold.

\begin{lemma}\label{reduceCases} 
Suppose $e\in \g$ is a nilpotent element with $e^{[p]}=0$, let $\tau$ be an associated cocharacter
and suppose $e$ is in the image of $\ad(e)^{p-1} \g(-2p+4)$.
Then the statements of Theorem \ref{necForW1s}(i) and (ii) hold.

Unless the Levi $\l$ associated to $e$ is of type $A_{p-1}$,
there is a unique $1$-space $\la f\ra\subseteq\g(-2p+4)$ such that $\ad(e)^{p-1}\la f\ra=\la e\ra$.
\end{lemma}

\begin{proof} 
This is a case by case check. To start with, we reduce the number of cases we must consider using Lemma \ref{e-candidates}. For each nilpotent orbit $e$ in good characteristic,
we may check to see whether it satisfies $e^{[p]}=0$, for example by looking at the tables in \cite{Law95}.
For each of these, we take an associated cocharacter $\tau$.
Helpfully, associated cocharacters are listed in the tables in \cite{LT11}.
One can then apply $\tau$ to each root vector and establish the dimensions of each piece of the associated grading.
Since $\g(-2p+4)$ is assumed to be non-zero, the possible cases for nilpotent orbits through $e$ are given in Table \ref{T:reduceCases}. 

Let us give an example: Suppose $e$ belongs to the orbit $F_4(a_2)$.
Then from \cite[p78]{LT11} we see that the associated cocharacter $\tau$ can be
taken to satisfy $\la \alpha_2, \tau \ra = 2 = \la \alpha_4,\tau \ra, \la \alpha_1, \tau \ra = 0 = \la \alpha_3,\tau \ra$.

Applying this to the negative of the highest root with coefficients $-2342$, we see this is in $\tau$-weight $-2\cdot3-2\cdot2=-10$.
Running through the remainder of the roots we can establish the possible $\tau$-weights which occur. Indeed $-10$ is the lowest weight.
Now if $2p-4\leq 10 \leq 2p-2$, we have that $p = 7$. 
This explains the entry for $F_4(a_2)$ in Table \ref{T:reduceCases}.

To complete the first part, one must check  for each case in Table \ref{T:reduceCases} to see whether the remaining hypothesis of the lemma is satisfied.
This is easily done using GAP, but one can of course do such calculations by hand. 
For a negative example, let us again consider the orbit $F_4(a_2)$ for $p=7$.
Here we may assume
$e = e_{1110} + e_{0001} + e_{0120} + e_{0100}$.
The space $\g(-10)$ is spanned by
$e_{-1342},e_{-2342}$.
Let $f = y_1 e_{-1342} + y_2 e_{-2342}$ be a generic element in $\g(-10)$. We compute
\begin{align*}
\ad(e)^{6} (f) = y_1 (2 \cdot e_{0011} + e_{1100} + 2 \cdot e_{0110} + 2 \cdot e_{1120})
+ y_2 ( e_{0001} + e_{1110} + e_{0120}),
\end{align*}
an expression which 
for no choice of $y_1,y_2$ is a nonzero multiple of $e$
(for instance it does not involve $e_{0100}$).
So $e$ does not satisfy the required conditions. 

For a positive example, let us take the simplest case where $\g=\g_2$, $p=7$ and $e$ is regular.
We may choose $e=e_{10}+e_{01}$. Corresponding to the associated cocharacter 
$\tau$ with weights $2$ on $10$ and $01$,
the $\tau$-weight space with weight $-2p+4=-10$ is occupied just by the span of the root vector which is the negative of the highest root,
namely $-32$. (Thus the uniqueness assertion follows immediately in this case.)
Now apply $(\ad (e_{10}+e_{01}))^6$ to $e_{-32}$.
One finds the answer is $10\cdot e_{10}+(-18)\cdot e_{01}$ which is $3\cdot e_{10}+3\cdot e_{01}$ modulo $7$
so that $e$ satisfies the hypotheses of the lemma.
Now notice that $7=h+1$ and $e$ is regular in $\g$ as required.

In case the dimension of the weight space $\g(-2p+4)$ is bigger than one-dimensional, say $\g(-2p+4)$ is spanned by $x_1,\dots,x_r$, then apply $(\ad e)^{p-1}$ to $\sum\lambda_ix_i$ and equate this to $e$.
This puts conditions on the $\lambda_i$ which are either uniquely satisfied, 
or one is in the case of one of the exceptions above.
\end{proof}

\begin{table}\begin{center}\begin{minipage}{150pt}\begin{center}\begin{tabular}{l l l}
$G$ & $p$ & $\OO$\\\hline\hline
$G_2$ & $7$ & $G_2$\\\hline
$F_4$ & $13$ & $F_4$\\
& $7$ & $F_4(a_2)$\\
& $7$ & $C_3$\\
& $7$ & $B_3$\\
& $5$ & $F_4(a_3)$\\
& $5$ & $C_3(a_1)$\\
& $5$ & $B_2$\\\hline
$E_6$ & $13$ & $E_6$\\
& $7$ & $E_6(a_3)$\\
& $7$ & $D_5(a_1)$\\
& $7$ & $A_5$\\
& $7$ & $D_4$\\
& $5$ & $A_4A_1$\\
& $5$ & $A_4$\\
& $5$ & $D_4(a_1)$\\
& $5$ & $A_3A_1$\\
& $5$ & $A_3$
\end{tabular}\end{center}\end{minipage}
\begin{minipage}{150pt}\begin{center}\begin{tabular}{l l l}
$G$ & $p$ & $\OO$\\\hline\hline
$E_7$ & $19$ & $E_7$\\
 & $13$ & $E_7(a_2)$\\
 & $13$ & $E_6$\\
 & $11$ & $E_7(a_3)$\\
 & $11$ & $D_6$\\
 & $7$ & $A_6$\\
 & $7$ & $E_7(a_5)$\\
 & $7$ & $E_6(a_3)$\\
 & $7$ & $D_6(a_2)$\\
 & $7$ & $D_5(a_1)A_1$\\
 & $7$ & $A_5A_1$\\
 & $7$ & $(A_5)'$\\
 & $7$ & $D_5(a_1)$\\
 & $7$ & $D_4A_1$\\
 & $7$ & $D_4$\\
 & $7$ & $(A_5)''$\\
 & $5$ & $A_4A_2$\\
 & $5$ & $A_4A_1$\\
 & $5$ & $A_3A_2A_1$\\
 & $5$ & $A_4$\\
 & $5$ & $A_3A_2$\\
 & $5$ & $D_4(a_1)A_1$\\
 & $5$ & $D_4(a_1)$\\
 & $5$ & $A_3A_1^2$\\
 & $5$ & $(A_3A_1)'$\\
 & $5$ & $(A_3A_1)''$\\
 & $5$ & $A_3$
\end{tabular}\end{center}\end{minipage}
\begin{minipage}{150pt}\begin{center}\begin{tabular}{l l l}
$G$ & $p$ & $\OO$\\\hline\hline
$E_8$ & $31$ & $E_8$\\
 & $19$ & $E_8(a_3)$\\
 & $19$ & $E_7$\\
 & $13$ & $E_8(a_5)$\\
 & $13$ & $E_8(b_5)$\\
 & $13$ & $D_7$\\
 & $13$ & $E_7(a_2)$\\
 & $13$ & $E_6A_1$\\
 & $13$ & $E_6$\\
 & $11$ & $E_8(a_6)$\\
 & $11$ & $D_7(a_1)$\\
 & $11$ & $E_7(a_3)$\\ 
 & $11$ & $D_6$\\  
 & $7$ & $A_6$\\
 & $7$ & $A_6A_1$\\
 & $7$ & $E_8(a_7)$\\
 & $7$ & $E_7(a_5)$\\
 & $7$ & $E_6(a_3)A_1$\\
 & $7$ & $D_6(a_2)$\\
 & $7$ & $D_5(a_1)A_2$\\
 & $7$ & $A_5A_1$\\
 & $7$ & $E_6(a_3)$\\
 & $7$ & $D_4A_2$\\
 & $7$ & $D_5(a_1)A_1$\\
 & $7$ & $A_5$\\
 & $7$ & $D_5(a_1)$\\
 & $7$ & $D_4A_1$\\
 & $7$ & $D_4$ 
\end{tabular}\end{center}\end{minipage}\end{center}\caption{\label{T:reduceCases}Cases to be checked in Lemma \ref{reduceCases}}\end{table}

The following lemma proves the existence assertion of Theorem \ref{necForW1s}
by guaranteeing the existence of at least one $p$-subalgebra of type $W_1$ whenever $(e,p)$ satisfy the conditions (i) and (ii) of Theorem \ref{necForW1s}.

\begin{lemma}\label{theW1sExist} 
Suppose $(e,p)$ satisfy the conditions (i) and (ii) of Theorem \ref{necForW1s} and let $\l$ be a Levi subalgebra in which $e$ is regular.
Then there exists a $p$-subalgebra $W\cong W_1\leq \l'$ with $\del$ represented by $e$.

\end{lemma}
\begin{proof} We have $\l'$ is simple. If $\l'$ is simple of classical type, then Lemma \ref{naturalreps} gives the result. The remaining cases occur when $e$ is regular in a Levi of exceptional type and $(p,\l')$ is one of $(7,G_2)$, $(13,F_4)$, $(13, E_6)$, $(19, E_7)$ or $(31,E_8)$. 

In each of these cases we may, without loss of generality, assume $\l'=\g$. Now Lemma \ref{reduceCases} implies that there is, up to scalars, a unique element $f\in \g(-2p+4)$ such that $e$ is a non-zero multiple of $(\ad e)^{p-1}f$; say $(\ad e)^{p-1}f=\lambda e$. (In each case, we may take $f$ to be the root vector corresponding to the negative of the highest root.) Now according to the standard basis of $W_1$, we have $(\ad \del)^{p-1}X^{p-1}\del=(p-1)!\del$. Thus, replacing $f$ with $f.(p-1)!/\lambda$ it suffices to check that there is a homomorphism $W_1\to\g$ obtained by sending \[(X^{p-1}\del,X^{p-2}\del,\dots,X\del,\del)\to\{f, 1/(p-1).[e,f], \dots, 1/(p-2)!(\ad e)^{p-2}f,1/(p-1)!(\ad e)^{p-1}f=e\}.\] For this it suffices to check that the latter elements satisfy the commutator and $p$-th power relations in $W_1$. This is a straightforward check using the commutator relations amongst basis elements of $\g$. These were performed in GAP.
\end{proof}

By analogy with the notion of a regular $A_1$ subalgebra, let us say that a $p$-subalgebra  $W_1$ of $\g$ is \emph{regular} if $\del$ is represented by a regular element of $\g$. (Of course, a regular $W_1$ then contains a regular $A_1$ subalgebra $\la \del,X\del,X^2\del\ra\cong \sl_2$.)
In order to show that non-regular $W_1$s are not maximal,
we will want to show that each normalises a non-trivial abelian subalgebra of $\g$.
For this most cases can be dealt with by showing that all possible modules whose composition factors coincide with those of the restriction $\g|W$ of the adjoint module $\g$ to $W$ must contain a trivial submodule. To do this we will compute in the next lemma all the possible composition factors.

\begin{lemma} Suppose $W\cong W_1$ is a $p$-subalgebra of $\g$ such that the element $\del\in W$ is represented by the nilpotent element $e$ such that $e$ is not of type $A_{p-1}$. Then the composition factors of $W$ on $\g$ are given in Table \ref{tComp}.\end{lemma}
\begin{proof}
Since $\g$ admits a grading $\g=\bigoplus\g(i)$ with $e=\del$ acting in degree $2$ we may invoke  Proposition \ref{wCompAlg}. Thus the composition factors of $\g|W$ can be computed recursively  by the weights of $X\del$ on the highest graded piece $\g(r)$ of $\g$. Since $e$ is not of type $A_{p-1}$ we have by Proposition \ref{findingcoch}(iii) that each $\g(i)$ is an eigenspace for $X\del$ with weight $i/2$. Thus we need only know the dimensions of each $\g(i)$, which is easy to compute in GAP.

Let us give an example. Suppose $e$ is a nilpotent element of $F_4$ of type $C_3$. Then since $e\in W$, we have $p=7$. The non-zero graded pieces are listed below (note that $\g(i)=\g(-i)$):

\begin{center}\begin{tabular}{c|c c c c c c c c c c c}
$i$ & $0$ & $1$ & $2$ & $3$ & $4$ & $5$ & $6$ & $7$ & $8$ & $9$ & $10$\\\hline
$\dim\g(i)$ & $6$ &$4$ & $3$ &$4$ & $2$ & $2$ & $2$ & $2$ & $1$ & $2$ & $1$\end{tabular}
\end{center}

Thus $\ad\del$ must have at least a one-dimensional kernel on $\g(10)$, $\g(6)$ and $\g(2)$ a two-dimensional kernel on $\g(9)$ and $\g(3)$. Thus amongst the composition factors of $L(G)|W$ we have must at least find $L(1)$, $L(3)$, $L(5)$, $L(5)^2$ and $L(1)^2$, respectively. Peeling off the weights corresponding to these submodules leaves just three weights in the zero weight space, which must correspond to three trivial composition factors. Thus the composition factors are $[\g|W]=L(1)^3,L(5)^3,L(3),k^3$.
\end{proof}

\begin{table}\begin{center}\begin{center}\begin{tabular}{l l l l}
$\g$ & $p$ & $\OO$ & $[\g|W]$\\\hline\hline
$G_2$ & $7$ & $G_2$ & $L(1),L(5)$\\\hline
$F_4$ & $13$ & $F_4$ & $L(1),L(5),L(7),L(11)$\\
& $7$ & $C_3$ & $L(1)^3,L(3),L(5)^3,k^3$\\
& $7$ & $B_3$ & $L(1),L(3)^5,L(5),k^3$\\
& $5$ & $B_2$ & $L(1),L(2)^4, L(3),L(4)^4,k^6$\\\hline
$E_6$ & $13$ & $E_6$ & $L(1),L(4),L(5),L(7),L(8),L(11)$\\
& $7$ & $A_5$ & $L(1)^3,L(2),L(3),L(4),L(5)^3,L(6)^2,k^3$\\
& $7$ & $D_4$ & $L(1),L(3)^8,L(5),k^8$\\
& $5$ & $A_3$ & $L(1),L(2)^5,L(3),L(4)^8,k^{11}$\\\hline
$E_7$ & $19$ & $E_7$ & $L(1),L(5),L(7),L(9),L(11),L(13),L(17)$ \\
 & $13$ & $E_6$ & $L(1),L(4)^3,L(5),L(7),L(8)^3,L(11),k^3$\\
 & $11$ & $D_6$ & $L(1),L(2)^2,L(3),L(5)^2,L(7),L(8)^2,L(9),L(10)^2,k^3$\\
 & $7$ & $(A_5)'$ & $L(1)^3,L(2)^3,L(3),L(4)^3,L(5)^3,L(6)^6,k^6$ \\
 & $7$ & $D_4$ & $L(1),L(3)^{14},L(5),k^{21}$\\
 & $7$ & $(A_5)''$ & $L(1),L(2)^7,L(3),L(4)^7,L(5),k^{14}$\\
 & $5$ & $A_3$ & $L(1),L(2)^7,L(3),L(4)^{16},k^{24}$ \\\hline
$E_8$ & $31$ & $E_8$ & $L(1),L(7),L(11),L(13),L(17),L(19),L(23),L(29)$\\
 & $19$ & $E_7$ & $L(1),L(4)^2,L(5),L(7),L(9),L(11),L(13),L(14)^2,L(17),L(18)^2,k^3$\\
 & $13$ & $D_7$ & $L(1)^3,L(3),L(4)^2,L(5),L(6)^3,L(7),L(8)^2,L(9),L(11)^3,L(12)^2,k^3$\\
 & $13$ & $E_6$ & $L(1),L(4)^7,L(5),L(7),L(8)^7,L(11),k^{14}$\\
 & $11$ & $D_6$ & $L(1),L(2)^4,L(3),L(5)^6,L(7),L(8)^4,L(9),L(10)^4,k^{10}$\\  
 & $7$ & $A_5$ & $L(1)^3,L(2)^7,L(3),L(4)^7,L(5)^3,L(6)^{14},k^{17}$\\
 & $7$ & $D_4$ & $L(1),L(3)^{26},L(5),k^{52}$
\end{tabular}\end{center}\end{center}\caption{Composition factors of subalgebras $W\cong W_1$ containing a nilpotent element of type $e$ not of type $A_{p-1}$}\label{tComp}\end{table} 

\begin{remark} Where $\OO$ is $A_{p-1}$, the potential composition factors of a corresponding $p$-subalgebra of type $W_1$ will in fact differ according to which toral element
$H_0$ represents $X\del = H + H_0$, since unlike in the case of Proposition \ref{findingcoch}(iii) they are not necessarily unique up to conjugacy. We have computed these also but since they require more in depth computations in GAP, we leave this to the Appendix \ref{appendix}, Proposition \ref{ptcomp}.
\end{remark}

To prove part (v) of Theorem \ref{necForW1s} we will want to see that in most cases, the composition factors listed above can only appear in modules in which there is a fixed vector, hence forcing a corresponding $W_1$ subalgebra into a parabolic. The following lemma gives a useful bound to ensure this.

\begin{lemma}\label{cfbound}
Suppose $p \geq 5$ and $V$ is a $W_1$-module with $[V:k]=n_0>0$, $[V:L(p-1)]=n_{-1}$ and $[V:L(1)]=[V:L(p-2)]=n_1$. Then if $V$ contains no trivial submodule, we have $n_0\leq 2n_{-1}+n_1$.\end{lemma}
\begin{proof} By induction on the number of composition factors. We cannot have $V$ irreducible. If $V$ contains $2$ composition factors then the module $V$ is uniserial with successive composition factors $k/L(1)$ or $k/L(p-1)$. Thus the result holds in both these cases.

Suppose the number of composition factors of $V$ is $r$. Then there is an irreducible submodule $S$, say, with $S=L(t)$, $t\neq 0$. First suppose $V/S$ contains no trivial submodules. Then by induction we have $n_0\leq 2n_{-1}+n_1$ if $t\neq 1,p-1$; $n_0\leq 2(n_{-1}-1)+n_1\leq 2n_{-1}+n_1$ if $t=p-1$; and $n_0\leq 2n_{-1}+n_1-1\leq 2n_{-1}+n_1$ if $t=1$, which proves the result in all these cases. Now suppose $V/S$ contains a trivial submodule $R$ of dimension $l$. Then from the exact sequence $0\to S\to V\to V/S\to 0$, taking the preimage $R'$ of $R$ in $V$ we have an exact sequence $0\to S\to R' \to R\to 0$ with no non-trivial $\g$-module map from $R\to R'$. By Lemma \ref{ext1forw1} there are no self-extensions of the trivial module, so $V/R'$ contains no trivial submodule and we may appeal to induction. Since there is no non-trivial $\g$-module map from $R\to R'$, it follows that $\dim \Ext^1_\g(k,S)\geq l$. Thus by Lemma \ref{ext1forw1}, $l\leq 2$. If $l=2$ then $S=L(p-1)$ and in the quotient of $V$ by $R'$ we have $n_0-2$ trivial composition factors and $n_{-1}-1$ composition factors isomorphic to $L(p-1)$. Thus by induction, the lemma holds for $V/R'$ and we have $n_0-2\leq 2(n_{-1}-1)+n_1$ as required. If $l=1$ then either $S=L(p-1)$ or $S=L(1)$. A similar argument by induction shows that the lemma holds again in each case.\end{proof}

The next lemma proves part (v) of Theorem \ref{necForW1s}.

\begin{lemma}\label{notreg=notmax} 
Let $W$ be a $p$-subalgebra of $\g$ with $\del$ represented by a nilpotent element $e$.
Then if $e$ is not a regular element,
$W$ normalises a non-trivial abelian subalgebra of $\g$, 
hence is not maximal.
\end{lemma}

\begin{proof} We use Lemma \ref{cfbound} together with Table \ref{tComp} when $e$ is not of type $A_{p-1}$. Assume $e$ is not regular or of type $A_{p-1}$, yet there is no trivial submodule in $\g|W$. Inspecting Table \ref{tComp} together with Lemma \ref{cfbound} one sees that this rules out most cases. 

We use a special argument in the case $(\g,p,e)$ is $(E_7,11,D_6)$ or $(E_8,19,E_7)$. Since $W$ is assumed to be a subalgebra of $\g$, its adjoint representation $L(p-2)$ must appear as a submodule of $\g|W$. Thus its unique composition factor $L(1)$ must appear in the head of the module. Take the quotient by $L(p-2)$ of the largest submodule of $\g$ not containing the composition factor $L(1)$; call this $M$. Then $M$ is still self-dual containing no composition factors of the form $L(1)$ or $L(p-2)$. By Lemma \ref{ext1forw1} the three trivial composition factors must appear in indecomposable subquotients of the form $k/L(p-1)$. Because $M$ is self-dual, there is a subquotient of the form $L(p-1)/k$. The composition factor $k$ in this subquotient must appear in a (different) subquotient of the form $k/L(p-1)$ and the submodule generated by vectors in $L(p-1)$ must contain both composition factors of this type. Since this applies to all composition factors of type $k$, we must have a subquotient of the form $L(p-1)/(k\oplus k\oplus k)$, but this is a contradiction by Lemma \ref{ext1forw1} as $\Ext_\g^1(L(p-1),k)$ is dimension just $2$.

This rules out all but the remaining eleven cases: 
\[(\g,p,e)=(E_6,5,A_4),\ (E_6,7,A_5),\ (E_7,5,A_4),\ (E_7,7,(A_5)'),\ (E_7,7,A_6),\ (E_8,13,D_7),\]
\[(E_8,7,A_6),\ (E_8,7,A_5),\ (E_7,5,A_3),\ ,(E_6,5,A_3),\ (F_4,5,B_2).\]

For these we perform an intricate series of direct checks in GAP to find a non-trivial
abelian subalgebra normalised by $W$ in $\g$ in all cases.
See Appendix \ref{appendix}.
\end{proof}

We now establish the remaining statements of Theorem \ref{necForW1s}.

\begin{lemma}\label{uniqueconjandmax}
\begin{itemize}
\item[(i)] There is a unique conjugacy class of regular $W_1$s. 
\item[(ii)] A regular $W_1$ in $\g$ is maximal if and only if $\g$ is not of type $E_6$.
\end{itemize}
\end{lemma}

\begin{proof}
Suppose the element $X^{p-1}\del$ is represented by a nilpotent element $f$.
By Proposition \ref{findingcoch}(iii) there is a cocharacter $\tau$ associated to $e=\del$
with $\Lie(\tau(\Gm)) = \la X\del \ra$. We have $X\del = -\frac{1}{2}d\tau(1)$.
Since $[X\del,X^{p-1}\del]=(p-2)X^{p-1}\del$, we get 
that $f$ is in the direct sum of the $\tau$-weight spaces congruent to $-2(p-2)=-2p+4$ modulo $p$.
Since $e$ is regular, all $\tau$-weights are even.
It follows that $f\in\g(-2p+4)\oplus\g(4)$.
Now using GAP, relations of the form $[f,e,f]=0$, $[f,e,\dots,e,f]=0$ quickly imply
that $f\in\g(-2p+4)$.
For example let $\g=G_2$, so that $p=7$ and $e=e_{10}+e_{01}$. We have $\g(-10)=\la e_{-32}\ra$ and $\g(4)=\la e_{11}\ra$. Then $f=\lambda_1e_{-32}+\lambda_2e_{11}$. Now $[f,e,f]=-6\lambda_2^2 e_{32}+2\lambda_1\lambda_2e_{-11}$. Thus we must have $\lambda_2=0$ as required.

Now, since $\dim\g(-2p+4)=1$, $f$ is unique up to scalars. Thus $W$ is uniquely 
determined by $e$, proving (i). 

For (ii), recall $F_4$ is a subalgebra of $E_6$. Under this embedding, one checks that a regular element in $F_4$ is also a regular element in $E_6$. By (i) there is a unique conjugacy class of subalgebras of type $W_1$, hence each is in a subalgebra of type $F_4$ and is not maximal. 

It remains to show that a regular $W_1$ is maximal in the remaining types.
Let $W$ be such a $p$-subalgebra and suppose it is not maximal. Then in the adjoint representation,
there is a minimal $W$-supermodule of $W$ which 
generates a proper subalgebra. The composition factors $[\g|W]$ were given in Table \ref{tComp}. One checks that for each type of $\g$, the dimension of the nullspace of $\del$ is equal to the number of composition factors in $[\g|W]$. Thus there is a basis $\mathcal B$ of $\tau$-weight vectors for the nullspace of $\del$ corresponding to each composition factor. Since in Table \ref{tComp} all the composition factors are pairwise distinct, a minimal supermodule of $W$ contains one of the elements in $\mathcal B$, $v$ say. It is then a computation in GAP to check that in all cases, $\la W,v\ra=\g$. This proves that $W$ is maximal.\end{proof}

\begin{proof}[\it Proof of Theorem A] We summarize by pointing out where in this section we have established the relevant statements in Theorem A.
Parts (i) and (ii) are found in Lemma \ref{reduceCases}; parts (iii) and (iv) are Lemma \ref{uniqueconjandmax}; part (v) is Lemma \ref{notreg=notmax}; the existence assertion is Lemma \ref{theW1sExist}.\end{proof}

\section{Other non-classical subalgebras: Proof of Theorem \ref{onlyW1s}} \label{sec:nonclass}

In this section we give the proof of Theorem \ref{onlyW1s}. We first give a reduction to finding subalgebras of $\g$ which are less than $p^3-3$-dimensional.

\begin{lemma}\label{nomelikyanorcontact} 
There is no proper subalgebra $\h$ of $\g$ of dimension $p^3-3$ or higher.\end{lemma}
\begin{proof} Since $5^3=125$ and $7^3=343$ the only possibility for such a subalgebra would be in $E_7$ when $p=5$. We may enlarge $\h$ to be a maximal $p$-subalgebra $\h'$. Now take a Weisfeiler filtration $(\g_{(n)})_{n\in\Z}$ with $\g_{(0)}=\h'$. By \cite[Proposition 3.1.1]{SF88} $\g_{(-1)}/\g_{(0)}$ is a faithful irreducible module for $\g_{(0)}/\g_{(1)}$. Since $\h$ is simple, and $\g_{(1)}$ is a nilpotent ideal of $\g_{(0)}$ we have $\h\cap \g_{(1)}=0$ and so $\g_{(-1)}/\g_{(0)}$ restricts to a faithful module for the image of $\h$ in $\g_{(0)}/\g_{(1)}$. 
The index of $\h$ in $\g$ is $\leq 11$; thus $\dim\g_{(-1)}/\g_{(0)}\leq 11$. This means that $\dim\h\leq 11^2=121<125-3=122$. But this is a contradiction.\end{proof}

%

Using the lemma, it follows from the Premet--Strade Classification \cite{PS06} that the the only simple subalgebras $\h$ of $\g$ of Cartan type which can possibly appear are as follows:
\begin{center}
\begin{tabular}{l l l l}
$\h$ & $p$ & $\dim \h$ & $\g$ \\ \hline \hline 
$W(1;1)$ & $p < \dim \g$ \\
$W(1;2)$ & $p \leq 13$ \\
$W(2;(1,1))$ & $p \leq 11$ \\
$H(2;(1,1))^{(2)}$ & $p \leq 13$ \\
$H(2;(1,1);\Phi(\tau))^{(1)}$ & $p \leq 13$ \\
$H(2;(1,1);\Phi(1))$ & $p \leq 13$ \\
\end{tabular}
\end{center}
(All dimensions of other simple Lie algebras are at least $p^3-1$. For the classification of rank one Hamiltonians, see \cite[6.3.10]{Str04}.)

\begin{proof}[Proof of Theorem \ref{onlyW1s}] First we reduce to the case that the $p$-closure $\h_p$ is the minimal $p$-envelope of $\h$. 

We work up inductively through the rank of $\g$. 
Since $5^2-2$ is bigger than $14$, we cannot have $\h\leq G_2$. Equally, since the smallest non-trivial representation of $\h$ has dimension no more than $p^2-2$, $\h$ cannot be contained in any classical $A$--$D$-type algebra of rank less than $12$.

Assume we have proved that $\h$ is not a subalgebra of any simple subalgebra of rank less than that of $\g$. Now if $\h\subseteq\h_p\subseteq\g$ with $\h_p$ not semisimple, then the radical $\Rad\h_p$ is a non-trivial ideal in $\h_p$. However $\h$ is also an ideal in $\h_p$, thus $\h$ and $\Rad(\h_p)$ are mutually normalising and have trivial intersection. Thus $\h$ centralises an element of the Lie algebra of $\g$ and hence is in a proper parabolic subalgebra of $\g$. Thus $\h$ projects to a simple subalgebra of a Levi subalgebra of this parabolic, which is of smaller rank, a contradiction.  

Hence we can assume that $\h_p$ is semisimple. 
In particular, its trivial centre is contained in $\h$.
By \cite[2.5.8(iii)]{SF88}, this forces $\h_p$ to be a minimal $p$-envelope of $\h$ as required.

Now we show that there is no $p$-subalgebra isomorphic to $\h_p$ where $\h_p$ is the minimal $p$-envelope of $\h$ where $\h$ is any of $H(2;(1,1))^{(2)}$ (with $\h=\h_p$), $H(2;(1,1);\Phi(\tau))^{(1)}$,  $H(2;(1,1);\Phi(1))$, or $W_2$. Since $H(2;(1,1))^{(2)}\leq W_2$ it suffices to deal with the first two cases. So assume, looking for a contradiction, that $\h\cong H(2;(1,1))^{(2)}$, $\h\cong H(2;(1,1);\Phi(\tau))^{(1)}$ or $\h\cong H(2;(1,1);\Phi(1))$ with $\h\leq\h_p\leq\g$.

By Lemma \ref{h2repsrest} and \ref{weirdh2repsrest} we know that there is a $p$-subalgebra $W\cong W_1$ with $\g|W$ having composition factors such that every $L(r)$ with $2\leq r\leq p-1$ appears the same number of times. But inspecting Table \ref{tComp} (for $\OO\neq A_{p-1}$) and Table \ref{t:ptComp} (when $\OO=A_{p-1}$) we find that no such $p$-subalgebra exists.

It remains to deal with the case $\h=W(1;2)$. 
This is the algebra given by basis $\{e_i:-1\leq i\leq p^2-2\}$ and multiplication
\begin{align*}
[e_i,e_j]=\begin{cases}
\left(\left(\begin{smallmatrix}i+j+1\\j\end{smallmatrix}\right) - 
\left(\begin{smallmatrix}i+j+1\\i\end{smallmatrix}\right)\right)
e_{i+j} &\text{ if }-1\leq i+j\leq p^2-2,\\
0 &\text{ otherwise,}\end{cases}
\end{align*}
(see for example \cite[p4]{Fel}), where
we put $\binom{j}{-1}=0$ for $j\geq 0$. 
Thus we check that the nilpotent endomorphism $\ad e_{-1}$ 
is in the image of $(\ad e_{-1})^{p^2-1}$.
We can assume again (by induction) that $\h_p$ is semsimple and that $\h_p$ is a minimal $p$-envelope of $\h$.
In particular the element in $\g$ representing $e_{-1}$ is nilpotent (cf.\ \emph{loc.\ cit.}).
Now one checks (e.g. the tables in \cite{LT11}) that the largest $r$ for which the space
$\g(-r)$ in the grading associated to $e$ is non-zero is $2h-2$.
Hence the largest $s$ for which $\g(2)$ is in the image of $(\ad e)^s$ is $h+1$.
Thus $p^2-1\leq h+1$, i.e. $p\leq \sqrt{h+2}$.
Now for a root system $\Phi$ of type $(G_2,F_4,E_6,E_7,E_8)$ we have $h+2$ is $(8,14,14,20,30)$. 
This implies $\p\leq (2,3,3,3,5)$, a contradiction as $p$ is a good prime.
Thus $\h$ does not appear as a subalgebra in $\g$. This finishes the proof of the theorem.

\end{proof}

\appendix
\section{The remaining cases from Lemma \ref{notreg=notmax}}
\label{appendix}

We have two jobs to perform in this section, both of which use GAP intensively. The first is to find the composition factors of $[\g|W]$ in the case that $W$ contains a nilpotent element of type $A_{p-1}$ representing $\del$. For the other,  recall that there are eleven cases of $(\g,p,\OO)$ for which we must check whether a $p$-subalgebra isomorphic to $W_1$
with $\del$ represented by a nilpotent element of type $\OO$ 
normalises a non-trivial abelian subalgebra of $\g$.

The cases are:

\begin{align} \label{firstcases}
\begin{split}
(\g,p,e) = &(F_4,5,B_2),\ (E_6,5,A_3),\ (E_6,7,A_5),\ (E_7,5,A_3),\\ &(E_7,7,(A_5)'),\ (E_8,7,A_5),\ (E_8,13,D_7),
\end{split}
\end{align} 
as well as the cases where $\OO = A_{p-1}$:
\begin{align}\label{secondcases}
(\g,p,e) = (E_6,5,A_4),\ (E_7,5,A_4),\ (E_7,7,A_6),\ (E_8,7,A_6).
\end{align}

Let us first calculate the composition factors of the restrictions $[\g|W]$ for $W$ containing a nilpotent element $e$ of type $A_{p-1}$ representing $\del$. For this we use the algorithm described in Proposition \ref{wCompAlg}. The data required is a grading $\g=\bigoplus\g(i)$ and the list of the  weights $\ell_i$ with multiplicities of $X\del$ on each $\g(i)$. In these cases we have many choices for a toral element $h$ representing $X\del$. However, by  Proposition \ref{findingcoch}(ii) we have that 
that it is of the form $H+H_0$ where $H\in\Lie(\tau(\Gm))$ and $H_0\in\g_e(0)\cap\im\ad e$.
We find $H$ by deriving the cocharacter $\tau$ given in \cite[p33]{LT11} and insisting
that it has the correct weight
$\tt [H,e] = [X\del ,e] = -e$.
See Table \ref{t:justHandH0} for our choices of $H$.

In cases \eqref{secondcases}, $H$ and $H_0$ commute, and $H$ is toral, so that $H_0$ is also toral.
Examining Table \ref{t:adecapg}, $\im \ad e \cap \g_e(0)$ is the Lie algebra of a connected reductive
algebraic group of rank $1$ so that $H_0$ is conjugate to a scalar multiple of some
fixed element. We produce this element in GAP by first taking a generic element $\tt w$ in the $-2$ weight
space for $\tau$, and then considering $\tt [e,w]$. This is a generic element in
$\im \ad e \cap \g(0)$, and insisting that it commutes with $\tt e$ and lies in the standard maximal
torus fixes a choice of $H_0$. We may now write $X\del = H + \lambda H_0$ with some scalar $\lambda$. (Note also that if $\{\alpha_1,\dots, \alpha_{p-1}\}$ are simple roots for the Levi subalgebra of type $A_{p-1}$ then $H_0$ lies in the centre of the corresponding $\sl_{p}$ so one can construct the element $H_0$ as $h_{\alpha_1}+2h_{\alpha_2}+\dots+(p-1)h_{\alpha_{p-1}}$.)
As $\lambda H_0$ is toral we must have $\lambda \in \F_p$.

\begin{table}\small \begin{tabular}{l|l|l}
$(\g,p,\OO)$ & $H$ & $H_0$\\\hline
$(E_6,5,A_4)$&$3 \cdot h_1 + 3\cdot h_2 + 2 \cdot h_3 + 2\cdot h_4 $  &
$ 2\cdot h_1 + 3\cdot h_2 + 4\cdot h_3 + h_4$\\
$(E_7,5,A_4)$&$3 \cdot  h_1 + 2 \cdot  h_3 + 2\cdot h_4 + 3\cdot h_2$&
$2\cdot h_1 + 3\cdot h_2 + 4\cdot h_3 + h_4$ \\
$(E_7,7,A_6)$&$4\cdot h_1 + 2\cdot h_3 + h_4 + h_5 + 2\cdot h_6 + 4\cdot h_7$&
$6\cdot h_1 +5\cdot h_3 + 4\cdot h_4 + 3\cdot h_5 + 2\cdot h_6 + h_7$ \\
$(E_8,7,A_6)$&$4\cdot h_2 + 2\cdot h_4 + h_5 + h_6 + 2\cdot h_7 + 4\cdot h_8$
&$ 6\cdot h_2 + 5\cdot h_4 + 4\cdot h_5 + 3\cdot h_6+ 2\cdot h_7 + h_8$.
\end{tabular}
\vspace{10pt}\caption{\label{t:justHandH0}Choices of $H$ and $H_0$}\end{table}

For the purposes of computation in GAP, we will need elements representing $e$, $H$ and $H_0$. In the canonical basis in GAP, these are given in Table \ref{t:HandH0}.

\begin{table}\tiny \begin{tabular}{l|l|l}
$(\g,p,\OO)$ & $e$ & $H$\\\hline
$(F_4,5,B_2)$&
$\tt B[3]+B[4]$&
$\tt B[51] + 3*B[52]$\\
$(E_6,5,A_3)$&
$\tt B[1]+B[3]+B[4]$&
$\tt B[73] + 3*B[75] + B[76]$\\
$(E_6,7,A_5)$&
$\tt B[1] + B[3] + B[4] + B[5] + B[6]$&
$\tt B[73] + 3*B[75] + 6*B[76] + 3*B[77] + B[78]$\\
$(E_7,5,A_3)$&
$\tt B[1]+B[3]+B[4]$&
$\tt B[127] + 3*B[129] + B[130]$\\
$(E_7,7,(A_5)')$&
$\tt B[1] + B[3] + B[4] + B[5] + B[6]$&
$\tt B[127] + 3*B[129] + 6*B[130] + 3*B[131] + B[132]$\\
$(E_8,7,A_5)$&
$\tt B[1]+B[3]+B[4]+B[5]+B[6]$&
$\tt B[241] + 3*B[243] + 6*B[244] + 3*B[245] + B[246]$\\
$(E_8,13,D_7)$&
$\tt B[2]+B[3]+B[4]+B[5]+B[6]+B[7]+B[8]$&
$\tt 7 * B[248] + 2 * B[247] + 11 * B[246] + 8 * B[245]$\\
&& $\tt+ 6 * B[244] + 9 * B[243] + 9 * B[242]$\\
$(E_6,5,A_4)$ &
$\tt B[1]+B[2]+B[3]+B[4]$ &
$\tt 3 * B[73] + 2 * B[75] + 2*B[76] + 3*B[74]$\\
$(E_7,5,A_4)$&
$\tt B[1]+B[2]+B[3]+B[4]$ &
$\tt 3 * B[127] + 2 * B[129] + 2*B[130] + 3*B[128]$\\
$(E_7,7,A_6)$&
$\tt B[1] + B[3] + B[4] + B[5] + B[6] + B[7]$ &
$\tt 4*B[127] + 2*B[129] + B[130] + B[131] + 2*B[132] + 4*B[133]$\\
$(E_8,7,A_6)$&
$\tt B[2]+B[4]+B[5]+B[6]+B[7]+B[8]$ &
$\tt 4*B[242] + 2*B[244] + B[245] + B[246] + 2*B[247] + 4*B[248]$\\
&&\\
$(\g,p,\OO)$ & $H_0$\\\hline
$(E_6,5,A_4)$ &
$\tt 2*B[73] + 3*B[74] + 4*B[75] + B[76]$ \\
$(E_7,5,A_4)$&
$\tt 2*B[127] + 3*B[128] + 4*B[129] + B[130]$ \\
$(E_7,7,A_6)$&
$\tt 6*B[127] +5*B[129] + 4*B[130] + 3*B[131] $\\&$\tt+ 2*B[132] + B[133]$ \\
$(E_8,7,A_6)$&
$\tt 6*B[242] + 5*B[244] + 4*B[245] + 3*B[246] $\\&$\tt+ 2*B[247] + B[248]$.
\end{tabular}
\vspace{10pt}\caption{\label{t:HandH0}Choices of $e$, $H$ and $H_0$}\end{table}

\begin{prop}\label{ptcomp}For the various choices of $\lambda\in\F_p$, Table \ref{t:ptComp} lists the possible composition factors of $\g|W$ where $W$ contains a nilpotent element $e$ of type $A_{p-1}$ representing $\del$, and a toral element $H+\lambda H_0$ representing $X\del$, for $H$ and $H_0$ in Table \ref{t:justHandH0}.\end{prop}

\begin{table}\begin{center}\begin{center}\begin{tabular}{l l l l l}
$\g$ & $p$ & $\OO$ & $\lambda$ & $[\g|W]$\\\hline\hline

$E_6$ & $5$ & $A_4$ & $0$ & $L(1)^5,L(2)^3,L(3)^5,L(4)^2,k^5$\\
&&& $1$ & $L(1),L(2)^5,L(3),L(4)^8,k^{11}$ \\
&&& $2$ & $L(1)^4,L(2),L(3)^4,L(4)^6,k^9$\\
&&& $3$ & $L(1)^4,L(2),L(3)^4,L(4)^6,k^9$\\
&&& $4$ & $L(1),L(2)^5,L(3),L(4)^8,k^{11}$\\

$E_7$ & $7$ & $A_6$ & $0$ & $L(1),L(2)^3,L(3)^5,L(4)^3,L(5),L(6)^6,k^6$\\
&& & $1$ & $L(1)^2,L(2)^3,L(3)^3,L(4)^3,L(5)^2,L(6)^6,k^6$\\
&& & $2$ & $L(1)^4,L(2)^2,L(3)^3,L(4)^2,L(5)^4,L(6)^4,k^4$\\
&& & $3$ & $L(1)^3,L(2)^3,L(3),L(4)^3,L(5)^3,L(6)^6,k^6$\\
&& & $4$ & $L(1)^3,L(2)^3,L(3),L(4)^3,L(5)^3,L(6)^6,k^6$\\
&& & $5$ & $L(1)^4,L(2)^2,L(3)^3,L(4)^2,L(5)^4,L(6)^4,k^4$\\
&& & $6$ & $L(1)^2,L(2)^3,L(3)^3,L(4)^3,L(5)^2,L(6)^6,k^6$\\

$E_7$ & $5$ & $A_4$ & $0$ & $L(1)^7,L(2)^9,L(3)^7,L(4)^2,k^{10}$ \\
&&& $1$ & $L(1),L(2)^7,L(3),L(4)^{16},k^{24}$ \\
&&& $2$ & $L(1)^8,L(2),L(3)^8,L(4)^8,k^{16}$ \\
&&& $3$ & $L(1)^8,L(2),L(3)^8,L(4)^8,k^{16}$ \\
&&& $4$ & $L(1),L(2)^7,L(3),L(4)^{16},k^{24}$ \\

$E_8$ &  $7$ & $A_6$ & $0$ & $L(1)^5,L(2)^3,L(3)^{13},L(4)^3,L(5)^5,L(6)^6,k^9$ \\
&&& $1$ & $L(1)^4,L(2)^5,L(3)^7,L(4)^5,L(5)^4,L(6)^{10},k^{13}$ \\
&&& $2$ & $L(1)^8,L(2)^4,L(3)^3,L(4)^4,L(5)^8,L(6)^8,k^{11}$ \\
&&& $3$ & $L(1)^3,L(2)^7,L(3),L(4)^7,L(5)^3,L(6)^{14},k^{17}$ \\
&&& $4$ & $L(1)^3,L(2)^7,L(3),L(4)^7,L(5)^3,L(6)^{14},k^{17}$ \\
&&& $5$ & $L(1)^8,L(2)^4,L(3)^3,L(4)^4,L(5)^8,L(6)^8,k^{11}$ \\
&&& $6$ & $L(1)^4,L(2)^5,L(3)^7,L(4)^5,L(5)^4,L(6)^{10},k^{13}$ 

\end{tabular}\end{center}\end{center}\caption{Composition factors of subalgebras $W\cong W_1$ containing a nilpotent element of type $e$ of type $A_{p-1}$, where
$X\del = H + \lambda H_0$}\label{t:ptComp}\end{table}

The rest of this appendix is dedicated to finishing the proof of  Theorem \ref{necForW1s}(v). We check directly in GAP whether such a subalgebra must fix a nonzero vector $v \in \g$. In most cases, we may find such a $v$ and are done. Our strategy is as follows:

\begin{itemize}

\item Set up a simple Lie algebra {\tt g} in GAP of the same type of $\g$ over the ring of polynomials 
$\Q[\tt x\_1,\dots, x\_{\dim\g}]$. Let {\tt B} be its basis. 
The GAP Data Library {\tt Lie Algebras} 
arranges {\tt B} so that {\tt B} is an array with {\tt B[$\dim\g-\rk\g+1$]},\dots,{\tt B[$\dim\g$]} a basis of toral elements
for a maximal torus of ${\tt g}$, and the remaining elements of ${\tt B}$ are root vectors for this torus, with the first $\rk\g$ of these being simple root vectors. These simple root vectors are normally in the Bourbaki ordering; the exception is type $F_4$, where one needs to apply a permutation.

\item From the tables in \cite{LT11}, set {\tt e} to be the nilpotent representative expressed 
in terms of the elements {\tt B[i]} and set {\tt T} to be an array whose entries are the coefficients of the cocharacter $\tau$
associated to {\tt e} in \cite{LT11}.
By the choice of cocharacter in \cite{LT11} we have that each element {\tt B[i]} is a weight vector for $\tau$.

\item Organise the vectors ${\tt B[i]}_{i\in \dim\g}$ into weight spaces for $\tau$.

\item Since $W$ contains a toral element representing $X\del$, we have by Proposition \ref{findingcoch}(ii)
that it is of the form $H+H_0$ where $H\in\Lie(\tau(\Gm))$ and $H_0\in\g_e(0)$.
We find $H$ by deriving the cocharacter $\tau$ given in \cite[p33]{LT11} and insisting
that it has the correct weight
$\tt [H,e] = [X\del ,e] = -e$.
See Table \ref{t:HandH0} for our choices of $H$.

By Proposition \ref{findingcoch}(iii), in cases \eqref{firstcases} where $e$ is not of type $A_{p-1}$ we have $H_0=0$, hence $X\del = H$.

In cases \eqref{secondcases}, $H$ and $H_0$ commute, and $H$ is toral, so that $H_0$ is also toral.
Examining Table \ref{t:adecapg}, $\im \ad e \cap \g_e(0)$ is the Lie algebra of a connected reductive
algebraic group of rank $1$ so that $H_0$ is conjugate to a scalar multiple of some
fixed element. We produce this element in GAP by first taking a generic element $\tt w$ in the $-2$ weight
space for $\tau$, and then considering $\tt [e,w]$. This is a generic element in
$\im \ad e \cap \g(0)$, and insisting that it commutes with $\tt e$ and lies in the standard maximal
torus fixes a choice of $H_0$. We may now write $X\del = H + \lambda H_0$ with some scalar $\lambda$. (Note also that if $\{\alpha_1,\dots, \alpha_{p-1}\}$ are simple roots for the Levi subalgebra of type $A_{p-1}$ then $H_0$ lies in the centre of the corresponding $\sl_{p}$ so one can construct the element $H_0$ as $h_{\alpha_1}+2h_{\alpha_2}+\dots+(p-1)h_{\alpha_{p-1}}$.)
As $\lambda H_0$ is toral we must have $\lambda \in \F_p$.
We use the choices of $H$ and $H_0$ from Table \ref{t:HandH0} and perform the following checks
for all $p$ possible values of $\lambda$.

\item Next we produce an element $\tt f$ that is a candidate for
$\frac{1}{2} X^2 \del$:
\begin{itemize}
\item Let {\tt f} be a generic element in $\tt g$.
By a generic element we mean an element of the form 
$\tt f:=\sum_{i}\tt x\_i.B[i]$.
\item We ensure that $\tt [e,f]=X\del$ and $\tt [X\del,f]=f$ by considering linear relations
among the $\tt x\_i$ resulting from these equations and substituting them in the coefficients
of $\tt f$.
\end{itemize}

\item The putative subalgebra $W$ contains, additionally, the element $X^3\del$ and, moreover $W$ is generated by $X^3\del$ and $\del$.
We perform a similar routine to the above to find an arbitrary element ${\tt ff}$ representing $\frac{1}{6}X^3\del$ on which $X\del$
has the correct weight $\tt [X\del,{\tt ff}] = 2{\tt ff}$. 
By substituting relations in both $\tt f$ and $\tt ff$ we force the relation
$\tt [e,ff]=f$.
(Note that this puts many constraints on $\tt ff$ but we do not attempt to guarantee that we have $\la{\tt ff,e}\ra\cong W$; indeed this will rarely be true.)

\item We look for a vector $\tt v \neq 0$ in $\g$ which is killed by $\tt e$ and $\tt ff$.
Since $W$ is generated by these elements, this will guarantee that $\tt v$ is a fixed vector for $W$.
Specifically:
\begin{itemize}
\item We form a generic element $\tt v$ from the basis vectors.
\item We compute $\tt[v,e]$. Forcing this to be zero puts constraints on the coefficients of $\tt v$.
\item We compute $\tt[X\del,v]$ and set this to be zero, putting more constraints on the coefficients.
\end{itemize}

\item Now consider the expression $\tt [ff,v] \in g$. Suppose that
$\tt x\_{i_1},\dots,x\_{i_r}$ are the indeterminates occuring in $\tt v$.
Now it turns out that the coefficients of $\tt [ff,v]$ in the basis $\tt B$ are
all linear expressions in the $\tt x\_{i_1},\dots,x\_{i_r}$. 
Thus there is a matrix $A$ whose entries are polynomials in the coefficients of $\tt ff$
and with $A \cdot \tt (x\_{i_1},\dots,x\_{i_r})^t = 0$ if and only if $\tt [ff,v] = 0$.

\item We proceed with doing row-reductions on $A$. If the rank of $A$ is strictly
smaller than $r$, we are done: $\tt v$ may be chosen to satisfy $\tt [ff,v]=0$.
This deals with all but the following cases: 

\begin{align*}
(\g,p,e) = 
&(E_6,7,A_5),\
(E_7,7,(A_5)'),\
(E_8,7,A_5),\
(E_8,13,D_7), \\
& (E_7,7,A_6)\text{ for }\lambda=1,2,3,4,5,6  \\
& (E_8,7,A_6)\text{ for }\lambda=2,5.
\end{align*}

\item We go on to consider the elements 
$\tt ff*f*f*f*f$ and $\tt ff*f*f*ff$, which both must vanish for $p=7$
(for the $p=13$ case we consider the element
$\tt ff*f*f*ff*f*f*ff*f*f$).
If there still are linear substitutions that may be read off from the coefficients of these
elements, we perform them on both $\tt f$ and $\tt ff$.
We next try to show that the remaining relations in $\tt ff*f*f*f*f=0=ff*f*f*ff$ again
force the rank of $A$ to be strictly less than $r$, by row-reducing $A$ one step at a time
and trying to substitute the above relations.

\item In the cases $(E_7,7,(A_5)')$ and $(E_8,7,A_5)$, we also use the following 
technique to reduce the number of indeterminates in $f$ and $ff$:
We consider root elements $y \in \Lie(C_e)$
with respect to the roots of $C_e$, the reductive part of the centraliser.
As $\ad(y)$ is nilpotent (in fact, $\ad(y)^4=0$ in all these cases), 
we obtain automorphisms $s_y(t) = \exp(t\cdot \ad(y))$ for $t \in k$ of $\g$, and
$s_y(t)$ satisfies $s_y(t)(e) =e$ and $s_y(t)(H)=H$.
Thus we may replace the pair $\tt (f,ff)$ by any pair
$(s_y(t)({\tt f}),s_y(t)({\tt ff}))$. Choosing $t$ and $y$ suitably, we may use this
to kill a number of coefficients in $\tt (f,ff)$.

\item After these steps, we succeed to find fixed vectors for all but the following cases:
\begin{align*}
&(E_7,7,A_6),\lambda=2,\tt x\_63*x\_108=0,\tt x\_63 \neq 0 \text{ or } \tt x\_108 \neq 0 \text{ and} \\ 
&(E_7,7,A_6),\lambda=5,\tt x\_41*x\_116=0,\tt x\_41 \neq 0 \text{ or } \tt x\_116 \neq 0.
\end{align*}

\item In the above two cases, we get that $\tt [ff,v]=0$ implies $\tt v=0$. This means
that the subalgebras $W$ corresponding to these cases do not fix a non-zero vector
in $\g$. However, let us not insist that $\tt [ff,v]=0$ but that
$\tt v*ff*ff*ff*ff*ff = 0$, by substituting linear equations in $v$. There still exist nonzero $v$
satisfying this relation.
In the cases $(E_7,7,A_6),\lambda=2,\tt x\_63=0,x\_108\neq 0$
and $\lambda=5,\tt x\_41\neq 0,x\_116=0$
we check with GAP that the subspace spanned by
$\tt v*f,v*f*f,\dots,v*f*f*f*f*f*f$ is an at most $6$-dimensional abelian subalgebra of $\tt g$ normalised by $W$, and we may
choose $\tt v$ such that it is nonzero.
In the cases $(E_7,7,A_6),\lambda=2,\tt x\_63\neq 0,x\_108=0$
and $\lambda=5,\tt x\_41=0,x\_116\neq 0$ we consider the subspace spanned by
$\tt v*ff,v*ff*ff*e,v*ff*ff,v*ff*ff*ff*e,v*ff*ff*ff,v*ff*ff*ff*ff*e,v*ff*ff*ff*ff$, which again is an abelian
subalgebra of $\tt g$, now at most $7$-dimensional. It is again normalised by $W$, and we may assume it nonzero.
Hence $W$ normalises a non-trivial abelian subalgebra in all cases.
\end{itemize}

Let us illustrate our procedure by giving a detailed example of our calculations in the case $(\g,p,\OO)=(E_8,7,A_6)$:

Let us first suppose that $\lambda=1$. We have:
\begin{verbatim}
e := B[2]+B[4]+B[5]+B[6]+B[7]+B[8]; 
T := [0,2,-10,2,2,2,2,2];
p := 7;
H := 4*B[242] + 2*B[244] + B[245] + B[246] + 2*B[247] + 4*B[248];
H0:= 6*B[242] + 5*B[244] + 4*B[245] + 3*B[246] + 2*B[247] + B[248];
Xd := H+H0;
\end{verbatim}

We start with completely generic $\tt f$ and $\tt ff$ and $\tt v$.
So for instance $\tt f = \sum_i x\_i B[i]$.
Now to ensure $\tt [e,f]=Xd$, we consider the difference 
$\tt [e,f]-Xd$:
{\footnotesize
\begin{verbatim}e*f - Xd =
(5*x_242+x_244)*v.2+(x_242+x_243+5*x_244+x_245)*v.4+(x_244+5*x_245+x_246)*v.5+(x_245+5\
*x_246+x_247)*v.6+(x_246+5*x_247+x_248)*v.7+(x_247+5*x_248)*v.8+(x_2+6*x_4)*v.10+(x_3)\
*v.11+(x_4+6*x_5)*v.12+(x_5+6*x_6)*v.13+(x_6+6*x_7)*v.14+(x_7+6*x_8)*v.15+(x_9)*v.16+(\
6*x_11)*v.17+(x_10+6*x_12)*v.18+(x_11)*v.19+(x_12+6*x_13)*v.20+(x_13+6*x_14)*v.21+(x_1\
4+6*x_15)*v.22+(6*x_16)*v.23+(x_16)*v.24+(x_17+6*x_19)*v.25+(x_18+6*x_20)*v.26+(x_19)*\
v.27+(x_20+6*x_21)*v.28+(x_21+6*x_22)*v.29+(x_23+6*x_24)*v.30+(x_24)*v.31+(x_25)*v.32+\
(x_25+6*x_27)*v.33+(x_26+6*x_28)*v.34+(x_27)*v.35+(x_28+6*x_29)*v.36+(x_30)*v.37+(x_30\
+6*x_31)*v.38+(x_31)*v.39+(x_32+x_33)*v.40+(x_33+6*x_35)*v.41+(x_34+6*x_36)*v.42+(x_35\
)*v.43+(x_37+x_38)*v.45+(x_38+6*x_39)*v.46+(x_39)*v.47+(x_40)*v.48+(x_40+x_41)*v.49+(x\
_41+6*x_43)*v.50+(x_44)*v.51+(x_45)*v.52+(x_45+x_46)*v.53+(x_46+6*x_47)*v.54+(x_48+x_4\
9)*v.55+(x_49+x_50)*v.56+(x_51)*v.57+(x_51)*v.58+(x_52+x_53)*v.59+(x_53+x_54)*v.60+(x_\
55)*v.61+(x_55+x_56)*v.62+(x_57)*v.63+(x_57+x_58)*v.64+(x_58)*v.65+(x_59)*v.66+(x_59+x\
_60)*v.67+(x_61+x_62)*v.68+(x_63)*v.69+(x_63+x_64)*v.70+(x_64)*v.71+(x_64+x_65)*v.72+(\
x_66+x_67)*v.73+(x_68)*v.74+(x_69+x_70)*v.75+(x_70+x_71)*v.76+(x_70+x_72)*v.77+(x_71+x\
_72)*v.78+(x_73)*v.79+(x_75+x_76)*v.80+(x_75+x_77)*v.81+(x_76)*v.82+(x_76+x_77+x_78)*v\
.83+(x_78)*v.84+(x_80+x_82)*v.85+(x_80+x_81+x_83)*v.86+(x_82+x_83)*v.87+(x_83+x_84)*v.\
88+(x_85)*v.89+(x_85+x_86+x_87)*v.90+(x_86+x_88)*v.91+(x_87+x_88)*v.92+(x_89+x_90)*v.9\
4+(x_90+x_91+x_92)*v.95+(x_92)*v.96+(x_93)*v.98+(x_94+x_95)*v.99+(x_95+x_96)*v.100+(x_\
97)*v.101+(x_98)*v.102+(x_99+x_100)*v.103+(x_101)*v.104+(x_102)*v.105+(x_103)*v.106+(x\
_104)*v.107+(x_105)*v.108+(x_107)*v.109+(x_108)*v.110+(x_109)*v.111+(6*x_110)*v.112+(6\
*x_111)*v.113+(6*x_114)*v.115+(6*x_115)*v.116+(6*x_116)*v.117+(6*x_117)*v.118+(6*x_118\
)*v.119+(6*x_119)*v.120+(6*x_130)*v.122+(6*x_131)*v.123+(x_130+6*x_132)*v.124+(x_132+6\
*x_133)*v.125+(x_133+6*x_134)*v.126+(x_134+6*x_135)*v.127+(x_135)*v.128+(6*x_136)*v.12\
9+(6*x_138)*v.130+(x_137+6*x_139)*v.131+(x_138+6*x_140)*v.132+(x_140+6*x_141)*v.133+(x\
_141+6*x_142)*v.134+(x_142)*v.135+(x_143+6*x_144)*v.136+(6*x_145)*v.137+(6*x_146)*v.13\
8+(x_145+6*x_147)*v.139+(x_146+6*x_148)*v.140+(x_148+6*x_149)*v.141+(x_149)*v.142+(6*x\
_150)*v.143+(x_150+6*x_151)*v.144+(6*x_152+6*x_153)*v.145+(6*x_154)*v.146+(x_153+6*x_1\
55)*v.147+(x_154+6*x_156)*v.148+(x_156)*v.149+(6*x_157+6*x_158)*v.150+(x_158+6*x_159)*\
v.151+(6*x_160)*v.152+(6*x_160+6*x_161)*v.153+(6*x_162)*v.154+(x_161+6*x_163)*v.155+(x\
_162)*v.156+(6*x_165)*v.157+(6*x_165+6*x_166)*v.158+(x_166+6*x_167)*v.159+(6*x_168+6*x\
_169)*v.160+(6*x_169+6*x_170)*v.161+(x_170)*v.163+(6*x_171)*v.164+(6*x_172+6*x_173)*v.\
165+(6*x_173+6*x_174)*v.166+(x_174)*v.167+(6*x_175)*v.168+(6*x_175+6*x_176)*v.169+(6*x\
_176)*v.170+(6*x_177+6*x_178)*v.171+(6*x_179)*v.172+(6*x_179+6*x_180)*v.173+(6*x_180)*\
v.174+(6*x_181+6*x_182)*v.175+(6*x_182)*v.176+(6*x_183+6*x_184)*v.177+(6*x_184+6*x_185\
)*v.178+(6*x_186+6*x_187)*v.179+(6*x_187)*v.180+(6*x_188)*v.181+(6*x_188)*v.182+(6*x_1\
89+6*x_190)*v.183+(6*x_190+6*x_191+6*x_192)*v.184+(6*x_192)*v.185+(6*x_193)*v.186+(6*x\
_193)*v.187+(6*x_194)*v.188+(6*x_195)*v.189+(6*x_195+6*x_196+6*x_197)*v.190+(6*x_196+6\
*x_198)*v.191+(6*x_197+6*x_198)*v.192+(6*x_199)*v.193+(6*x_200+6*x_201)*v.195+(6*x_200\
+6*x_202+6*x_203)*v.196+(6*x_201+6*x_203)*v.197+(6*x_203+6*x_204)*v.198+(6*x_205+6*x_2\
06)*v.200+(6*x_206)*v.201+(6*x_205+6*x_207)*v.202+(6*x_206+6*x_207+6*x_208)*v.203+(6*x\
_208)*v.204+(6*x_209+6*x_210)*v.205+(6*x_210+6*x_211)*v.206+(6*x_210+6*x_212)*v.207+(6\
*x_211+6*x_212)*v.208+(6*x_214)*v.209+(6*x_214+6*x_215)*v.210+(6*x_215)*v.211+(6*x_215\
+6*x_216)*v.212+(6*x_218)*v.213+(6*x_219)*v.214+(6*x_219+6*x_220)*v.215+(6*x_220)*v.21\
6+(6*x_221)*v.217+(6*x_222)*v.218+(6*x_223)*v.219+(6*x_223)*v.220+(6*x_224)*v.221+(6*x\
_225)*v.222+(6*x_226)*v.223+(6*x_227)*v.224+(6*x_228)*v.225+(6*x_229)*v.227+(6*x_230)*\
v.228+(6*x_231)*v.229+(x_232)*v.230+(x_233)*v.231+(x_235)*v.234+(x_236)*v.235+(x_237)*\
v.236+(x_238)*v.237+(x_239)*v.238+(x_240)*v.239+(x_122+4)*v.242+(x_124)*v.244+(x_125+2\
)*v.245+(x_126+3)*v.246+(x_127+3)*v.247+(x_128+2)*v.248\end{verbatim}}
We may repeat linear substitutions of the form $\tt x\_244 = -5x\_242$ and so on to make this expression vanish.
Similarly, we manage to ensure after multiple substitutions that
$\tt [Xd,f]=f, [Xd,ff]=2ff,[e,ff]=f, [e,v]=0$ and $\tt [Xd,v] = 0$.
This leaves us with the following:
{\footnotesize\begin{verbatim}
f := 
(x_42)*v.42+(x_50)*v.48+(6*x_50)*v.49+(x_50)*v.50+(x_54)*v.52+(6*x_54)*v.53+(x_54)*v.54\
+(x_96)*v.94+(6*x_96)*v.95+(x_96)*v.96+(3)*v.122+(5)*v.125+(4)*v.126+(4)*v.127+(5)*v.12\
8+(4*x_198)*v.189+(3*x_198)*v.190+(4*x_198)*v.191
ff :=
(x_36+x_42)*v.34+(x_36)*v.36+(x_50)*v.40+(5*x_50)*v.41+(4*x_50)*v.43+(x_54)*v.45+(5*x_5\
4)*v.46+(4*x_54)*v.47+(x_91+3*x_96)*v.89+(6*x_91+5*x_96)*v.90+(x_91)*v.91+(x_96)*v.92+(\
4)*v.130+(4)*v.132+(6)*v.133+(2)*v.134+(5)*v.135+(x_139)*v.137+(x_139)*v.139+(x_144)*v.\
143+(x_144)*v.144+(3*x_198)*v.195+(2*x_198)*v.196+(6*x_198)*v.197+(x_198)*v.198
v :=
(x_1)*v.1+(6*x_100)*v.99+(x_100)*v.100+(x_121)*v.121+(x_123)*v.123+(x_129)*v.129+(x_185\
)*v.183+(6*x_185)*v.184+(x_185)*v.185+(x_241)*v.241+(6*x_248)*v.242+(5*x_248)*v.244+(4*\
x_248)*v.245+(3*x_248)*v.246+(2*x_248)*v.247+(x_248)*v.248
\end{verbatim}}

Now $\tt v$ has the indeterminates $\tt [ 1, 100, 121, 123, 129, 185, 241, 248 ]$ and
hence $r=8$ is the number of indeterminates.
From $\tt [ff,v]$, 
{\footnotesize\begin{verbatim}
ff*v =
(5*x_50*x_123+5*x_54*x_129+4*x_96*x_185+5*x_100*x_198)*v.34+(4*x_50*x_123+4*x_54*x_129+\
6*x_96*x_185+4*x_100*x_198)*v.36+(x_50*x_241+5*x_50*x_248+x_54*x_121)*v.40+(5*x_50*x_24\
1+4*x_50*x_248+5*x_54*x_121)*v.41+(4*x_50*x_241+6*x_50*x_248+4*x_54*x_121)*v.43+(x_1*x_\
50+6*x_54*x_241+5*x_54*x_248)*v.45+(5*x_1*x_50+2*x_54*x_241+4*x_54*x_248)*v.46+(4*x_1*x\
_50+3*x_54*x_241+6*x_54*x_248)*v.47+(3*x_91*x_248+2*x_96*x_248+5*x_100)*v.89+(4*x_91*x_\
248+x_96*x_248+5*x_100)*v.90+(3*x_91*x_248+3*x_100)*v.91+(3*x_96*x_248+3*x_100)*v.92+(6\
*x_1*x_144+3*x_54*x_185+6*x_139*x_241+2*x_139*x_248+3*x_123)*v.137+(6*x_1*x_144+3*x_54*\
x_185+6*x_139*x_241+2*x_139*x_248+3*x_123)*v.139+(4*x_50*x_185+6*x_121*x_139+x_144*x_24\
1+2*x_144*x_248+3*x_129)*v.143+(4*x_50*x_185+6*x_121*x_139+x_144*x_241+2*x_144*x_248+3*\
x_129)*v.144+(5*x_198*x_248+3*x_185)*v.195+(x_198*x_248+2*x_185)*v.196+(3*x_198*x_248+6\
*x_185)*v.197+(4*x_198*x_248+x_185)*v.198,
\end{verbatim}}
we need the matrix $A$ as described above to be
\begin{verbatim}
gap> A :=
[ [ 0, 5*x_198, 0, 5*x_50, 5*x_54, 4*x_96, 0, 0 ], 
  [ 0, 4*x_198, 0, 4*x_50, 4*x_54, 6*x_96, 0, 0 ], 
  [ 0, 0, x_54, 0, 0, 0, x_50, 5*x_50 ],
  [ 0, 0, 5*x_54, 0, 0, 0, 5*x_50, 4*x_50 ], 
  [ 0, 0, 4*x_54, 0, 0, 0, 4*x_50, 6*x_50 ],
  [ x_50, 0, 0, 0, 0, 0, 6*x_54, 5*x_54 ], 
  [ 5*x_50, 0, 0, 0, 0, 0, 2*x_54, 4*x_54 ],
  [ 4*x_50, 0, 0, 0, 0, 0, 3*x_54, 6*x_54 ],
  [ 0, 5, 0, 0, 0, 0, 0, 3*x_91+2*x_96 ],
  [ 0, 5, 0, 0, 0, 0, 0, 4*x_91+x_96 ], 
  [ 0, 3, 0, 0, 0, 0, 0, 3*x_91 ],
  [ 0, 3, 0, 0, 0, 0, 0, 3*x_96 ], 
  [ 6*x_144, 0, 0, 3, 0, 3*x_54, 6*x_139, 2*x_139 ], 
  [ 6*x_144, 0, 0, 3, 0, 3*x_54, 6*x_139, 2*x_139 ], 
  [ 0, 0, 6*x_139, 0, 3, 4*x_50, x_144, 2*x_144 ], 
  [ 0, 0, 6*x_139, 0, 3, 4*x_50, x_144, 2*x_144 ],
  [ 0, 0, 0, 0, 0, 3, 0, 5*x_198 ], 
  [ 0, 0, 0, 0, 0, 2, 0, x_198 ],
  [ 0, 0, 0, 0, 0, 6, 0, 3*x_198 ], 
  [ 0, 0, 0, 0, 0, 1, 0, 4*x_198 ] ].
\end{verbatim}
After row reductions we obtain 
\begin{verbatim}
gap> Ared := 
[ [ x_50, 0, 0, 0, 0, 0, 6*x_54, 5*x_54 ],
  [ 0, 3, 0, 0, 0, 0, 0, 3*x_91 ], 
  [ 0, 0, x_54, 0, 0, 0, x_50, 5*x_50 ], 
  [ 0, 0, 0, x_50, x_54, 5*x_96, 0, 6*x_91*x_198 ], 
  [ 0, 0, 0, 0, 3*x_54, 4*x_50*x_54, x_50*x_139+x_54*x_144,
                                      5*x_50*x_139+2*x_54*x_144 ],
  [ 0, 0, 0, 0, 0, 1, 0, 4*x_198 ],
  [ 0, 0, 0, 0, 0, 0, 0, x_91+6*x_96 ], 
  [ 0, 0, 0, 0, 0, 0, 0, 0 ],
  ...
  [ 0, 0, 0, 0, 0, 0, 0, 0 ] ]
\end{verbatim}
and this clearly has rank less than or equal to $7$. As $7<8=r$, we are done.

In contrast to this, for $\lambda = 2$, we obtain in the end the matrix
\begin{verbatim}
gap> A :=
[ [ 0, 0, 6*x_62, 6*x_67, 2*x_84, 0, 0 ],
[ 0, 0, 5*x_62, 5*x_67, 4*x_84, 0, 0 ], 
[ 0, 6*x_67, 0, 5*x_84, 0, 6*x_62, 2*x_62 ], 
[ 0, 2*x_67, 0, 4*x_84, 0, 2*x_62, 3*x_62 ], 
[ 6*x_62, 0, 2*x_84, 0, 0, x_67, 2*x_67 ], 
[ 2*x_62, 0, 3*x_84, 0, 0, 5*x_67, 3*x_67 ], 
[ 0, 0, 0, 0, 0, 0, x_84 ], 
[ 0, 0, 0, 0, 0, 0, 3*x_84 ], 
[ 0, 0, 0, 0, 0, 0, 5*x_84 ], 
[ 0, 0, 0, 0, 0, 0, 3*x_84 ], 
[ 0, 0, 6*x_112, x_113, 0, 0, 3*x_106 ], 
[ 0, 6*x_113, 0, 0, 0, x_112, x_112 ], 
[ 6*x_112, 0, 0, 0, 0, 6*x_113, x_113 ], 
[ x_159, 0, 5, 0, 0, x_155, 5*x_155 ], 
[ 6*x_159, 0, 2, 0, 0, 6*x_155, 2*x_155 ], 
[ 6*x_159, 0, 2, 0, 0, 6*x_155, 2*x_155 ], 
[ 0, x_155, 0, 5, 0, 6*x_159, 5*x_159 ], 
[ 0, 6*x_155, 0, 2, 0, x_159, 2*x_159 ], 
[ 0, 6*x_155, 0, 2, 0, x_159, 2*x_159 ], 
[ 0, 0, x_159, 6*x_155, 5, 0, 3*x_178 ], 
[ 0, 0, 6*x_159, x_155, 2, 0, 4*x_178 ], 
[ 0, 0, 0, 0, 0, 0, x_234 ] ]
\end{verbatim}
and it turns out after row-reductions that a priori $A$ may have rank $7$ and $r=7$.
But now we consider $\tt ff*f*f*f*f$:
{\footnotesize\begin{verbatim}
ff*f*f*f*f;
(2*x_84^2*x_234+4*x_62*x_155+4*x_67*x_159+6*x_84*x_178+6*x_42)*v.2+(2*x_84^2*x_234+4*x_\
62*x_155+4*x_67*x_159+6*x_84*x_178+6*x_42)*v.4+(2*x_84^2*x_234+4*x_62*x_155+4*x_67*x_15\
9+6*x_84*x_178+6*x_42)*v.5+(2*x_84^2*x_234+4*x_62*x_155+4*x_67*x_159+6*x_84*x_178+6*x_4\
2)*v.6+(2*x_84^2*x_234+4*x_62*x_155+4*x_67*x_159+6*x_84*x_178+6*x_42)*v.7+(2*x_84^2*x_2\
34+4*x_62*x_155+4*x_67*x_159+6*x_84*x_178+6*x_42)*v.8+(x_62*x_84^2*x_234+5*x_62^2*x_155\
+5*x_62*x_67*x_159+4*x_62*x_84*x_178+x_36*x_62+6*x_42*x_62)*v.74+(x_67*x_84^2*x_234+5*x\
_62*x_67*x_155+5*x_67^2*x_159+4*x_67*x_84*x_178+x_36*x_67+6*x_42*x_67)*v.79+(2*x_84^3*x\
_234+3*x_62*x_84*x_155+3*x_67*x_84*x_159+x_84^2*x_178+2*x_36*x_84+5*x_42*x_84)*v.89+(5*\
x_84^3*x_234+4*x_62*x_84*x_155+4*x_67*x_84*x_159+6*x_84^2*x_178+5*x_36*x_84+2*x_42*x_84\
)*v.90+(2*x_84^3*x_234+3*x_62*x_84*x_155+3*x_67*x_84*x_159+x_84^2*x_178+2*x_36*x_84+5*x\
_42*x_84)*v.91+(4*x_84^4*x_234+4*x_62*x_84^2*x_155+4*x_67*x_84^2*x_159+6*x_84^3*x_178+2\
*x_36*x_84^2+3*x_42*x_84^2+x_62*x_113+x_67*x_112+5*x_84*x_106)*v.120+(3*x_84*x_234)*v.2\
00+(4*x_84*x_234)*v.201+(x_84*x_234)*v.202+(3*x_84*x_234)*v.203+(4*x_84*x_234)*v.204+(6\
*x_67*x_234)*v.213+(6*x_62*x_234)*v.217.
\end{verbatim}}
We see that there still are some linear substitutions possible. After performing
these, we obtain
{\footnotesize
\begin{verbatim}
ff*f*f*f*f =
(5*x_62*x_84^2*x_234)*v.74+(5*x_67*x_84^2*x_234)*v.79+(3*x_84^3*x_234)*v.89+(4*x_84^3*x\
_234)*v.90+(3*x_84^3*x_234)*v.91+(x_84^4*x_234+x_84*x_112*x_155+6*x_84*x_113*x_159+x_62\
*x_113+x_67*x_112)*v.120+(3*x_84*x_234)*v.200+(4*x_84*x_234)*v.201+(x_84*x_234)*v.202+(\
3*x_84*x_234)*v.203+(4*x_84*x_234)*v.204+(6*x_67*x_234)*v.213+(6*x_62*x_234)*v.217.
\end{verbatim}}
Let us first suppose that $\tt x\_{234} \neq 0$. Then the last equation implies
$\tt x\_62=x\_67=x\_84=x\_112=x\_113 = 0$ and this is enough to ensure that the rank
of $A$ is at most $4$ and we are done.

On the other hand, suppose $\tt x\_{234} = 0$. Then the expression for 
$\tt ff*f*f*f*f$ implies that
\begin{align} \label{eq:finaltweak}
\tt
x\_84*x\_112*x\_155+6*x\_84*x\_113*x\_159+x\_62*x\_113+x\_67*x\_112 = 0.
\end{align}
Now we start applying row-reductions to $A$ while distinguishing further subcases.
For instance we may consider the subcase $\tt x\_159,x\_155,x\_113,x\_112 \neq 0$ 
and take $\tt x\_159$ as the pivot
entry for our first row-reduction in $A$. After further row-reductions, we end up with
the following matrix:
{\footnotesize \begin{verbatim}
Ared := 
[ [ x_159, 0, 5, 0, 0, x_155, 5*x_155 ], 
[ 0, x_155, 0, 5, 0, 6*x_159, 5*x_159 ], 
[ 0, 0, x_159, 6*x_155, 5, 0, 3*x_178 ], 
[ 0, 0, 0, 5*x_113, 0, x_112*x_155+6*x_113*x_159, x_112*x_155+5*x_113*x_159 ], 
[ 0, 0, 0, 0, x_112*x_113, 2*x_112^2*x_155^2+3*x_112*x_113*x_155*x_159+2*x_113^2*x_159^2, 
            2*x_112^2*x_155^2+5*x_113^2*x_159^2+2*x_112*x_113*x_178 ], 
[ 0, 0, 0, 0, 0, 4*x_84*x_112*x_113^2*x_155*x_159^2+3*x_84*x_113^3*x_159^3+4*x_62*x_113^3*
            x_159^2+4*x_67*x_112*x_113^2*x_159^2, 6*x_84*x_112*x_113^2*x_155*x_159^2+
            4*x_84*x_113^3*x_159^3+3*x_62*x_113^3*x_159^2+x_67*x_112*x_113^2*x_159^2 ],
[ 0, 0, 0, 0, 0, 3*x_84*x_112*x_155^2+4*x_84*x_113*x_155*x_159+3*x_62*x_113*x_155+
            3*x_67*x_112*x_155, 3*x_84*x_112*x_155^2+x_84*x_113*x_155*x_159+x_62*x_113*
            x_155+3*x_67*x_112*x_155 ], 
[ 0, 0, 0, 0, 0, 5*x_84*x_112*x_113^2*x_155*x_159^2+2*x_84*x_113^3*x_159^3+5*x_62*
            x_113^3*x_159^2+5*x_67*x_112*x_113^2*x_159^2, 4*x_84*x_112*x_113^2*x_155*
            x_159^2+5*x_84*x_113^3*x_159^3+2*x_62*x_113^3*x_159^2+
            3*x_67*x_112*x_113^2*x_159^2 ], 
[ 0, 0, 0, 0, 0, x_84*x_112^2*x_113*x_155^2*x_159+5*x_84*x_112*x_113^2*x_155*x_159^2+
            x_84*x_113^3*x_159^3+x_62*x_112*x_113^2*x_155*x_159+6*x_62*x_113^3*x_159^2+
            x_67*x_112^2*x_113*x_155*x_159+6*x_67*x_112*x_113^2*x_159^2, 
            x_84*x_112^2*x_113*x_155^2*x_159+6*x_84*x_113^3*x_159^3+5*x_62*x_112*x_113^2*
            x_155*x_159+x_62*x_113^3*x_159^2+x_67*x_112^2*x_113*x_155*x_159+5*x_67*x_112*
            x_113^2*x_159^2+x_84*x_112*x_113^2*x_159*x_178 ],
[ 0, 0, 0, 0, 0, 2*x_84*x_112^2*x_113*x_155^2*x_159+3*x_84*x_112*x_113^2*x_155*x_159^2+
            2*x_84*x_113^3*x_159^3+2*x_62*x_112*x_113^2*x_155*x_159+
            5*x_62*x_113^3*x_159^2+2*x_67*x_112^2*x_113*x_155*x_159+
            5*x_67*x_112*x_113^2*x_159^2, 2*x_84*x_112^2*x_113*x_155^2*x_159+
            5*x_84*x_113^3*x_159^3+3*x_62*x_112*x_113^2*x_155*x_159+2*x_62*x_113^3*x_159^2+
            2*x_67*x_112^2*x_113*x_155*x_159+3*x_67*x_112*x_113^2*x_159^2+
            2*x_84*x_112*x_113^2*x_159*x_178 ],
[ 0, 0, 0, 0, 0, 2*x_84*x_112*x_155^2+5*x_84*x_113*x_155*x_159+2*x_62*x_113*x_155+
            2*x_67*x_112*x_155, 2*x_84*x_112*x_155^2+3*x_84*x_113*x_155*x_159+
            3*x_62*x_113*x_155+2*x_67*x_112*x_155 ],
[ 0, 0, 0, 0, 0, 0, x_84 ]]
\end{verbatim}}
Now we may recognise multiples of the left hand side of \eqref{eq:finaltweak} above in this matrix in
its sixth column. This means that $\tt Ared[i][6] = 0$ for all $i \geq 6$ and hence $A$ has
rank at most $6$ and we are done.
Continuing in this way, for each subcase we may row-reduce
$A$ to a matrix $A'$, where we may substitute the equation \eqref{eq:finaltweak} to deduce that 
the rank of $A'$ is strictly less than $r$.
The cases for the remaining choices
of $\lambda$ are dealt with similarly.

Finally we give an example for the use of automorphisms of the form
$\exp(t\cdot \ad(y))$, as mentioned in the steps above.
We consider the case $(E_8,7,A_5)$. After the first steps, we get the following
pair $\tt (f,ff)$.
\begin{verbatim}
f := 
(x_39)*v.39+(x_47)*v.47+(x_52)*v.51+(x_52)*v.52+(x_96)*v.96+(x_97)*v.97+(x_101\
)*v.101+(x_118)*v.118+v.121+(3)*v.123+(6)*v.124+(3)*v.125+v.126+(x_134)*v.127+\
(x_142)*v.135+(x_145+6*x_146)*v.137+(x_145+6*x_146)*v.138+(x_199)*v.194+(6*x_2\
00)*v.195+(6*x_206)*v.201+(6*x_233)*v.232;
ff := 
(6*x_39)*v.35+(6*x_47)*v.43+(x_48+2*x_52)*v.44+(x_48+x_52)*v.45+(x_48)*v.48+(x\
_96)*v.92+(x_97)*v.93+(x_101)*v.98+(6*x_118)*v.117+(6)*v.129+(3)*v.131+(4)*v.1\
32+v.133+(x_134)*v.134+(x_142)*v.142+(2*x_145+6*x_146)*v.143+(x_145)*v.145+(x_\
146)*v.146+(x_199)*v.199+(x_200)*v.200+(x_206)*v.206+(x_233)*v.233
\end{verbatim}

Here $C_e$ is of type $G_2A_1$. The root $y = \tt B[8]$ and its negative
$z = \tt B[128]$ belong to $C_e$. The corresponding automorphisms
${\tt s1} := s_y(t)=\exp(t\cdot \ad(y))$, 
${\tt s2} :=s_z(t)=\exp(t\cdot \ad(z))$ have the
following effect on $\tt (f,ff)$:

\begin{verbatim}
s1(f) =
(x_39)*v.39+(x_39*t+x_47)*v.47+(x_52)*v.51+(x_52)*v.52+(x_96)*v.96+(x_97)*\
v.97+(x_97*t+x_101)*v.101+(x_118)*v.118+v.121+(3)*v.123+(6)*v.124+(3)*v.12\
5+v.126+(6*x_142*t+x_134)*v.127+(x_142)*v.135+(x_145+6*x_146)*v.137+(x_145\
+6*x_146)*v.138+(x_199)*v.194+(x_206*t+6*x_200)*v.195+(6*x_206)*v.201+(6*x\
_233)*v.232
s1(ff) =
(6*x_39)*v.35+(6*x_39*t+6*x_47)*v.43+(x_48+2*x_52)*v.44+(x_48+x_52)*v.45+(\
x_48)*v.48+(x_96)*v.92+(x_97)*v.93+(x_97*t+x_101)*v.98+(6*x_118)*v.117+(6)\
*v.129+(3)*v.131+(4)*v.132+v.133+(6*x_142*t+x_134)*v.134+(x_142)*v.142+(2*\
x_145+6*x_146)*v.143+(x_145)*v.145+(x_146)*v.146+(x_199)*v.199+(6*x_206*t+\
x_200)*v.200+(x_206)*v.206+(x_233)*v.233
s2(f) =
(x_47*t+x_39)*v.39+(x_47)*v.47+(x_52)*v.51+(x_52)*v.52+(x_96)*v.96+(x_101*\
t+x_97)*v.97+(x_101)*v.101+(x_118)*v.118+v.121+(3)*v.123+(6)*v.124+(3)*v.1\
25+v.126+(x_134)*v.127+(6*x_134*t+x_142)*v.135+(x_145+6*x_146)*v.137+(x_14\
5+6*x_146)*v.138+(x_199)*v.194+(6*x_200)*v.195+(x_200*t+6*x_206)*v.201+(6*\
x_233)*v.232
s2(ff) =
(6*x_47*t+6*x_39)*v.35+(6*x_47)*v.43+(x_48+2*x_52)*v.44+(x_48+x_52)*v.45+(\
x_48)*v.48+(x_96)*v.92+(x_101*t+x_97)*v.93+(x_101)*v.98+(6*x_118)*v.117+(6\
)*v.129+(3)*v.131+(4)*v.132+v.133+(x_134)*v.134+(6*x_134*t+x_142)*v.142+(2\
*x_145+6*x_146)*v.143+(x_145)*v.145+(x_146)*v.146+(x_199)*v.199+(x_200)*v.200+\
(6*x_200*t+x_206)*v.206+(x_233)*v.233
\end{verbatim}

We see that if the coefficient $\tt x\_39$ is nonzero in $\tt (f,ff)$,
we may assume that $\tt x\_47$ is also nonzero after applying $\tt s1$
with $\tt t$ chosen such that $\tt x\_39*t+x\_47 \neq 0$.
But then we may apply $\tt s2$ to ensure that $\tt x\_39=0$.
Thus we may assume from the outset that $\tt x\_39=0$.

Now the automorphism $\tt s2$ may only
be applied again in the case $\tt x\_47=0$, otherwise it destroys the
property $\tt x\_39=0$. However, we have a total of $14$ root elements 
$y \in \Lie(C_e)$ to use for this process. 
Combined with the methods from before we again succeed in showing
that the rank of the corresponding matrix $A$ is less than $r$.

{\footnotesize
\bibliographystyle{amsalpha}
\bibliography{bib}}

\end{document}